\documentclass[11pt]{amsart}
\usepackage{newlfont,amsmath,mathrsfs,enumerate,amssymb,array,amscd,xcolor}

\usepackage{multirow}
\usepackage{amsmath, amssymb}
\usepackage{amscd}
\usepackage{amsthm}

\usepackage[normalem]{ulem}
\usepackage{mathabx}

\usepackage{parskip}

\newcommand{\Ext}{{\mathrm{Ext}}}

\newcommand{\rder}{{\mathrm{der}}}

\newcommand{\Ind}{{\mathrm{Ind}}}

\newcommand{\ind}{{\mathrm{ind}}}
\newcommand{\mini}{{\mathrm{mini}}}

\newcommand{\Spin}{{\mathrm{Spin}}}
\newcommand{\GSpin}{{\textnormal{GSpin}}}
\newcommand{\almostGSpin}{{}^*{\textnormal{GSpin}}}
\newcommand{\SL}{{\mathrm{SL}}}
\newcommand{\Sp}{{\mathrm{Sp}}}
\newcommand{\GL}{{\mathrm{GL}}}

\newcommand{\St}{\mathrm{St}}

\newcommand{\Hom}{{\mathrm{Hom}}}

\newtheorem{lemma}{Lemma}[section] 

\newtheorem{prop}[lemma]{Proposition}
\newtheorem{thm}[lemma]{Theorem}
\newtheorem{cor}[lemma]{Corollary}
\newtheorem{rmk}[lemma]{Remark}

\begin{document}


\title{$\SL_3$ dual pairs}

\author{Petar Baki\'c, Wee Teck Gan and Gordan Savin}

\address{P. B.: Department of Mathematics, University of Utah, Salt Lake City, UT}\email{bakic@math.utah.edu}
\address{W.T.G.:   Department of Mathematics, National University of Singapore, 10 Lower Kent Ridge Road
Singapore 119076} \email{matgwt@nus.edu.sg}
\address{G. S.: Department of Mathematics, University of Utah, Salt Lake City, UT}\email{savin@math.utah.edu}

\begin{abstract}
We study the exceptional theta correspondence for dual pairs of the form $\SL_3\times G$ inside $p$-adic groups of type $E_n$, $n=6,7,8$. We establish Howe duality for lifts of all Whittaker-generic unitarizable representations of $\SL_3$. Our dual pair is in a see-saw relationship with a dual pair of the form $G_2 \times H$ such that $G/H$ is a rank-two symmetric space. Using the theta correspondence for $\SL_3 \times G$, we prove a multiplicity-one result for $H$-invariant functionals on representations of $G$.
\end{abstract}
\maketitle
{
\section*{Introduction} 

\renewcommand{\thelemma}{\Alph{lemma}}

Let $F$ be a $p$-adic field. A simple, split Lie algebra $\mathfrak e$ of type $E_n$, $n=6,7,8$,  has a $\mathbb Z/3\mathbb Z$-grading 
\[ 
\mathfrak e=   \mathfrak{sl}_3 \oplus \mathfrak g \oplus V_3\otimes J \oplus V_3^{\ast}\otimes J^{\ast}  
\] 
where $V_3$ is the standard representation of $\SL_3$ and, if $G$ is the split simply connected group corresponding to $\mathfrak g$, then $J$ is a representation 
of $G$. The pairs $(G,J)$ are tabulated here: 
\begin{equation}
\label{table}
\begin{array}{c||c|c|c}
n & 6 & 7 & 8\\
\hline
G & \SL_3\times\SL_3 & \SL_6 & E_6 \\
\hline 
J &  V_3\otimes V_3 & \wedge^2 V_6 & V_{27} \\
\hline 
H & \SL_3 & \Sp_6 & F_4
\end{array}
\end{equation}
Note that $G$ acts on $J$ with a small kernel $\mu_m$, $m=3,2,1\ (=9-n)$. 
The group $G$ preserves a symmetric trilinear form $(x,y,z)$ on $J$.  
In fact, $G/\mu_m$ is the connected component of the group of linear transformations of $J$ preserving the trilinear form. 
The Lie bracket on $\mathfrak e$ can be written down using the trilinear form \cite{Rum97}. 
If we pick an element $e\in J$ such that $(e,e,e)\neq 0$ (a generic element) then $J$ can be turned into a Freudenthal-Jordan algebra with the identity $e$ 
and the (cubic) norm $N_J$ such that $6N_J(x)=(x,x,x)$ and $e$ is the identity of $J$. The stabilizer of $e$ in $G$ is a group isomorphic to $H$,  listed in the table above. 
The stabilizers of two generic elements $e\in J $ and $f\in J$  are conjugate in $G$ if and only if $N_J(e)/N_J(f)$ is a cube; see Proposition \ref{P:conjugating_H}.  

The purpose of this paper is to study the theta lift from $\SL_3$ to $G$ arising from the minimal representation $\Pi$ of $E_n$, with an application to rank-two symmetric spaces 
$G/H$. More precisely, fix an embedding of $H$ into $G$. Since $H$ fixes a line in both $J$ and $J^{\ast}$, the centralizer of $H$ in $\mathfrak e$ is 
\[ 
\mathfrak e^H=   \mathfrak{sl}_3 \oplus  V_3 \oplus V_3^{\ast} \cong \mathfrak g_2,
\] 
which we recognize as the exceptional Lie algebra of type $G_2$. Thus we have a see-saw of dual pairs 
\begin{equation} \label{eq_seesaw1}
\arraycolsep=1.5pt\def\arraystretch{1.35}
\begin{array}{ccc}
G_2 & & G \\
| & \bigtimes & | \\
 \SL_3  & & H  
\end{array}
\end{equation} 
This see-saw plays a key role for both the theta correspondence and its relationship with the analysis of symmetric spaces $G/H$.  
More precisely, if $(U,\psi)$ is a Whittaker pair for $\SL_3$, then 
\[ 
\Pi_{U,\psi} \cong C_c(G/H). 
\] 
We note that conjugacy classes of Whittaker pairs for $\SL_3$ are also parameterized by classes of cubes in $F^{\times}$. Thus as $\psi$ varies, so does the group $H$ in $G$. 
With that caveat in mind, the above identity, via the standard see-saw argument, implies that, for any finite length representation $\pi$ of $G$, 
\[ 
\Hom_H(\pi,\mathbb C) \cong \Hom_{U}(\Theta(\pi), \psi), 
\] 
where $\Theta(\pi)=(\Pi\otimes\pi^{\vee})_G$ is the so-called big theta lift of $\pi$. In other words, understanding whether $\pi$ has an $H$-invariant functional reduces to computing $\Theta(\pi)_{U,\psi}$.  

We now highlight some of the results obtained in this paper.  The first one is Theorem \ref{thm_main1}: 
\begin{thm}  Let $\sigma$ be an irreducible unitarizable $\psi$-generic representation of $\SL_3$. Then $\Theta(\sigma)$ is an irreducible and unitarizable representation 
of $G$. If $\sigma'$ is another irreducible unitarizable $\psi'$-generic representation of $\SL_3$, then $\Theta(\sigma)\cong \Theta(\sigma')$ implies $\sigma \cong \sigma'$. 
\end{thm} 

Recall that an irreducible representation $\pi$ of $G$ is hermitian if $\bar\pi\cong \pi^{\vee}$. Clearly all unitarizable (in particular tempered) representations are hermitian. 
Other hermitian representations are described in terms of the corresponding standard modules, \cite{BM03}.  
The second result is Theorem \ref{thm_main2}, a multiplicity-one statement for $H$-invariant functionals:

\begin{thm}  Let $\pi$ be an irreducible hermitian representation of $G$. Then $\dim \Hom_H(\pi,\mathbb C) \leq 1$. It is one if and only if $\pi$ is a theta lift of a $\psi$-generic representation $\sigma$ of $\SL_3$, for the Whittaker character $\psi$ determined by the conjugacy class of $H$. 
\end{thm} 

We derive both of these multiplicity-free results using the method of period ping-pong, see \S \ref{subs:pp}. As far as making the correspondence explicit, if $\sigma$ is tempered, then we determine the parameter of $\Theta(\sigma)$. In particular, it satisfies the expectations of BZSV \cite{BZSV} for $H$-distinguished representations.. We also have a rather complete description of $\Theta(\sigma)$ for non-tempered $\sigma$. For example, if $\sigma$ is (full) induced from the Borel subgroup, then $\Theta(\sigma)$ is the full degenerate principal series induced from a co-rank $2$ parabolic subgroup with Levi of type $A_1\times A_1\times A_1$ if $G=\SL_6$, and type $D_4$ if $G=E_6$. We note that the symplectic period in $G=\SL_6$ has been studied in \cite{Mi14} and \cite{OS07}; the main novelty of this theorem is thus the $E_6/F_4$ case.

\subsection*{Overview}
We begin in Section \ref{S:embed} by discussing the possible embeddings of $H$ in $G$ and $\SL_3$ in $G_2$, as seen in diagram \eqref{eq_seesaw1}. In Section \ref{S:periods} we compute the key periods, including the Whittaker coinvariants of the minimal representation, as well as $\Pi_H=\Theta(1_H)$, the lift of the trivial representation of $H$ to $G_2$ and its restriction to $\SL_3$. Section \ref{S:JMSL3} is devoted to a computation of the Jacquet module of the minimal representation with respect to a maximal parabolic in $\SL_3$; this is then used in Section \ref{S:noncusp} to compute the lifts of non-cuspidal representations. Section \ref{S:3SL3} is a self-contained treatment of the correspondence in $E_6$; the rest of the paper focuses on the correspondences in $E_7$ and $E_8$. 

First, in Section \ref{S:support}, we narrow down the possibilities for cuspidal support of lifts from $\SL_3$ to $G$. The relevant Jacquet modules are then computed in Sections \ref{S:Siegel} and \ref{S:heisenberg}. Having done the necessary preparations, we are ready to play period ping-pong: we use it to prove our main results in Section \ref{S:tabletennis}. To complete the picture, we compute the lifts of tempered representations in Section \ref{S:parameters}. Appendix \ref{S:orbits} features a discussion of $\SL_3$ orbits on $G_2$ flag varieties; this is used in crucial period computations in \S\ref{S:periods}. In Appendix \ref{S:mini}, we summarize the theta correspondence between $\GL_2$ and a similitude spin group, which features in Jacquet module computations of Section \ref{S:JMSL3} and thus plays a prominent role throughout the paper.

\subsection*{Notation and conventions}
We fix a local non-archimedean field $F$ of characteristic $0$ and let $\bar F$ denote its algebraic closure. We often use the same symbol to denote a reductive group as well as its $F$-points; this abuse of notation is harmless, as there is never any ambiguity.

Unless otherwise stated, we work in the category of smooth (complex) representations. If $\pi$ is a representation of a $p$-adic group $G$, and $\psi$ a character of a subgroup $H\subset G$, we write $\pi_{H,\psi}$ for the space of twisted coinvariants of $\pi$ with respect to $(H,\psi)$; if $\psi$ is trivial, we write simply $\pi_H$. We use $\ind$ to denote compact (and in particular unnormalized parabolic) induction, and $\Ind$ to denote normalized parabolic induction.

We often refer to the big/full theta lift, as defined above. As is customary, we refer to the maximal semisimple quotient of big theta as the small theta lift.

Finally, to avoid introducing excessive notation, we recycle certain symbols (like $\Theta$, $P=MN$, etc.) throughout the paper. This is always explicitly announced, so there is no fear of confusion.
}

\section{Some embeddings} \label{S:embed}
Let $x\in J$ be an element of non-zero norm. Recall that the stabilizer of $x$ in $G$ is a group isomorphic to $H$, shown in Table \eqref{table}. These stabilizers are parameterized by classes of cubes:
\begin{prop} \label{P:conjugating_H}
Let $x,y\in J$ have non-zero norm. Then the stabilizers in $G$ of $x$ and $y$ (two groups isomorphic to $H$) are conjugate in $G$ if and only 
if $N_J(x)/N_J(y)$ is a cube.  
\end{prop} 
\begin{proof} 
In \cite{SW15}, $G$-orbits on $J$ are classified over $\mathbb Z$. Extrapolating to fields, it follows that $G$ acts transitively on elements of a fixed non-zero norm. 
Since the stabilizers of $x$ and $ax$ are the same for $a\in F^{\times}$, and since $N_J(ax)=a^3 N_J(x)$, we have one direction of the statement. We use Galois cohomology to complete the proof. The normalizer of $H$ in $G$ is $H\times\mu_3$, hence the orbits correspond to points in the kernel of the map 
\[ 
H^1(F, H\times \mu_3) \rightarrow H^1(F,G). 
\] 
Since $H$ and $G$ are simply connected, their Galois cohomology vanishes. This implies the kernel of the above map is $H^1(F,\mu_3)=F^{\times}/(F^{\times})^3$. 
\end{proof} 


%


Let $H_a\subseteq G$ be the subgroup fixing an element $x\in J$ of norm $a$. As explained in the introduction, we have a see-saw of dual pairs 
\begin{equation} \label{eqseesaw}
\arraycolsep=1.5pt\def\arraystretch{1.35}
\begin{array}{ccc}
G_2^{a} & & G \\
| & \bigtimes & | \\
 \SL_3  & & H_a.  
\end{array}
\end{equation} 
Here we use the superscript in $G_2^a$ to emphasize that this subgroup isomorphic to $G_2$ depends on $H_a$; for $a=1$ we omit the superscript and write simply $G_2$. We fix an isomorphism $f_a : G_2^a \to G_2$. Precomposing this isomorphism with the embedding $\SL_3 \to G_2^a$ from the above diagram, we get an embedding $i_a: \SL_3\rightarrow G_2$. As $x$ varies, so does $i_a$. Let us answer how:

\begin{prop}
\label{p:SL3toG2}The conjugacy class of the embedding  $i_a\colon \SL_3\to G_2$ from see-saw diagram \eqref{eqseesaw} is obtained from the one for $a=1$ by pre-conjugating 
$i_1\colon \SL_3 \to G_2$ with an element $g$ in $\GL_3$ of determinant $a$.  
\end{prop}

\begin{proof}  Let $\mu_m$, $m=3,2,1$,  be the center of simply connected $E_n$. Then $\mu_m \subset G$, and let $\bar G=G/\mu_m$. 
This group acts faithfully on $J$ and it is in fact  (the connected component of)  the group of linear transformations of $J$ preserving the cubic form. 
Since $\SL_3\cap \bar G=\mu_3$, the adjoint $E_n$ contains a group 
\[ 
\SL_3\times_{\mu_3} \bar G. 
\] 
Let us describe the groups of $F$-points. To that end, let $G\bar G$ be (the connected component of) 
 the group of similitudes of the cubic norm $N_J$, and let $s:G\bar G\rightarrow F^{\times}$ be the 
similitude character. The group $\bar G$ is the kernel of $s$. 
Let $(\GL_3\times G\bar G)^{\det}$ be the subgroup consisting of pairs $(g,h)$ such that $\det(g)\cdot s(h)=1$. Then this group acts naturally on 
\[ 
\mathfrak{sl}_3 \oplus \mathfrak g \oplus V_3\otimes J \oplus V_3^{\ast}\otimes J^{\ast}  
\] 
with the kernel given by $\GL_1$ embedded by $\Delta(a)=(a,a^{-1})$ into $(\GL_3\times G\bar G)^{\det}$. Thus 
\[ 
(\SL_3\times_{\mu_3} \bar G)(F)= (\GL_3(F)\times G\bar G(F))^{\det}/\Delta(F^{\times}). 
\] 

Now we can finish the proof. An element $(g, h)$ of $(\GL_3\times G\bar G)^{\det}$ such that 
$\det(g)= s(h^{-1})=a$  conjugates the above see-saw diagram (\ref{eqseesaw}) to the one with $a=1$. 
However, the resulting vertical arrow from $\SL_3$ is pre-conjugated by $g$.
\end{proof} 
We have a similar situation for conjugacy classes of Whittaker characters for $\SL_3$, as they are also parameterized by classes of cubes.  Indeed, 
if we fix an additive character $\psi$, then any Whittaker character is of the form 
\[ 
 \left(\begin{matrix} 
1 & x & z  \\
0 & 1 &  y \\
0 & 0  &  1 \\ 
\end{matrix}\right) \mapsto \psi(ax+by). 
\] 
Its invariant is $a/b$, considered in the class of cubes in $F^{\times}/(F^{\times})^3$. It is easy to see that two Whittaker functionals are conjugate if and only if 
the invariants are the same. Note that conjugating the Whittaker character by $g\in \GL_3$ changes its invariant by $\det(g)$.  

Summarizing the above discussion, we see that the following are all torsors under $F^\times / (F^\times)^3$:
\begin{enumerate}[(i)]
    \item $G$-conjugacy classes of embeddings $H \to G$;
    \item $G_2$-conjugacy classes of embeddings $\SL_3 \to G_2$;
    \item $\SL_3$-conjugacy classes of Whittaker data $(U,\psi)$ for $\SL_3$.
\end{enumerate}
Proposition \ref{p:SL3toG2} shows that the see-saw diagram gives an isomorphism between (i) and (ii). In Section \ref{S:periods}, we show that the theta correspondence gives rise to an isomorphism between (i) and (iii).

\section{Some periods} \label{S:periods}
Here we compute the key periods. These computations rely on concrete models for the minimal representation, given in Theorems \ref{prop:msE6} and \ref{prop:msE8}. Similar computations have been carried out in various settings, so we occasionally omit the details and instead point to relevant references. 

\subsection{Computing Whittaker coinvariants} \label{subs_Whittaker_coinv}
For $a\in F^{\times}$, let $H_a\subseteq G$ be the stabilizer of an element in $J$ of norm $a$. Fix an additive character $\psi_0$ of $F$.
Let  $(\psi_a,U)$ be a Whittaker datum for $\SL_3$ such that $U$ is the subgroup of unipotent upper triangular matrices, and 
 $\psi_a(u)=\psi_0(ax+y)$, where $x,y$ are the (1,2) and (2,3) entries of $U$.
Let $\Pi$ be the minimal representation.  Then  $\Pi_{U,\psi_a}\simeq C_c(\omega_a)$, 
where $\omega_a$ is the set of $x\in J$ such that $N_J(x)=a$. This is completely analogous, in fact easier, than \cite[page 213, Lemma 2.9]{GrS98}. 
In \cite{SW15}, $G$ orbits on $J$ are classified over $\mathbb Z$. Extrapolating to fields, it follows that $G$ acts transitively on $\omega_a$. Hence $\omega_a=G/H_a$ and 
\[ 
\Pi_{U,\psi_a}\simeq C_c(G/H_a). 
\]

\subsection{Computing $\Theta(1_H)$} 
We drop the subscript $a$, as it plays no role in this section. 
For the dual pair $G_2 \times H$ we compute $\Theta(1_H)=\Pi_H$. 
 Let $P=MN\subset G_2$ be the Heisenberg maximal parabolic subgroup, that is, the stabilizer of a 2-dimensional nil-subspace $V_2$ (see Appendix \ref{S:orbits}). 
Then $M\simeq \GL(V_2)$, and the modular character $\rho_N^{1/2}$ is $|\det|^{3/2}$.   
Let $I(s)=\Ind_P^{G_2}(|\det|^s)$ denote the degenerate series induced from the character $|\det|^s$ of $\GL_2\simeq M$. For $s\geq 0$, this representation reduces only for $s=1/2$ and $3/2$; at $3/2$ it has the trivial representation as a unique quotient, see \cite{Mu97}.

\begin{prop} \label{P:Theta_1} 
For $G_2\times H$ in $E_n$, $n=6,7,8$, 
$\Theta(1_H)$ is isomorphic to $I(s)$ for $s=1/2, 3/2,7/2$, respectively. 
\end{prop} 
\begin{proof} Using Jacquet functors \cite{SW15}, one checks that $\Theta(1_H)$ is a quotient of $I(s)$ with the 
listed $s$. Recall, from \cite{HMS},  that
 $M$-orbits of generic characters of $N$ are parameterized by cubic \'etale algebras $E$. Let $\psi_E$ denote a character in the corresponding orbit. If $\pi$ is a non-trivial irreducible representation of $G_2$, then 
 $\pi_{\psi_E}\neq 0$ for some $E$ \cite[Prop. 3.4]{GS23}. Since the trivial representation of $G_2$ appears only in $I(3/2)$ as a unique 
 quotient, to finish the proof, it suffices to show that 
 \[ 
 I(s)_{N,\psi_E}\simeq  \Theta(1_H)_{N,\psi_E} 
 \] 
for all $E$. The former space is one-dimensional by Bruhat decomposition, so it suffices to prove that the latter is one-dimensional for all $E$.  
To that end, recall \cite[page 213, Lemma 2.9]{GrS98} that $\Pi_{N,\psi_E} \simeq C_c(H/H_E)$ for a certain reductive subgroup $H_E\subset H$. 
Now the desired one-dimensionality follows from the following sequence of isomorphisms, 
\[ 
 \Theta(1_H)_{N,\psi_E} \simeq (\Pi_H)_{N,\psi_E}\simeq (\Pi_{U,\psi_E})_H\simeq C_c(H/H_E)_H \simeq \mathbb C, 
\]
where the first one is the definition of $\Theta(1_H)$, the second is switching the order of coinvariants, and the last follows from Frobenius reciprocity. 
\end{proof}

\subsection{Restricting $I(s)$ to $\SL_3$}   \label{subs:boundary}
We compute the restriction of  the degenerate principal  series $I(s)= \Ind_P^{G_2}(|\det|^s)$ to $\SL_3$. We realize $G_2$ and $\SL_3\subseteq G_2$ as in Appendix \ref{S:orbits}. 
In particular, we have fixed the embedding $\SL_3\subseteq G_2$.

Recall, from Proposition \ref{P:orbtis_3} (and subsequent discussion) that $\SL_3$ acts on $G_2/P$ with two orbits, such that the stabilizer of $V_2$ in the closed orbit is $B$, and the stabilizer of $V'_2$ in the closed orbit is $B_0 = T_0U_0$. Fix the Whittaker datum $(U,\psi_1)$ for $\SL_3$; note that $U_0$ is in the kernel of $\psi_1$. 

Before the general case, consider the unnormalized $\ind_P^{G_2}(1)$  (=  functions on $G_2/P$). We claim that, as an $\SL_3(F)$-module, it has 
a filtration with three subquotients (from top to bottom) 
\begin{itemize} 
\item $\ind_B^{\SL_3}(1)$;
\item $\ind_B^{\SL_3}(\rho_U^{1/4} \cdot C_c(F^{\times}))$;
\item $\ind_U^{\SL_3}(\psi_1)$ (Gelfand--Graev).
\end{itemize} 
The representation in the first bullet comes from the closed orbit.  
The open orbit gives a sub $\ind_{B_0}^{\SL_3}(1)$ which we shall analyze using induction in stages, where we first induce from 
$B_0$ to $T_0U$. Identify $U/U_0 \cong F$ using either the (1,2) or (2,3) entry of $U$. Thus $\ind^{T_0U}_{B_0}(1) \cong C_c(F)$. Apply the Fourier transform. Then 
$C_c(F^{\times})$ is a submodule and the quotient is given by evaluating functions at 0. As $T_0U$-modules we have an exact sequence 
\[ 
0\rightarrow \ind_{U}^{T_0U}(\psi_1) \rightarrow \ind^{T_0U}_{B_0}(1)  \rightarrow \mathbb C \cdot |a| \rightarrow 0 
\] 
where the last is a character of $T_0U$ given by 
\[ 
 \left(\begin{matrix} 
a & x & z  \\
0 & 1 &  y \\
0 & 0  & a^{-1} \\ 
\end{matrix}\right) \mapsto |a |, 
\] 
 a restriction of $\rho_U^{1/4}$. Now the promised structure of the last two bullets follows by induction in stages. For  the second bullet, 
we first induce the character $|a|$ to $B$ and use $\ind_{T_0}^T(1) \cong C_c(F^{\times})$ so that 
\[ 
 \left(\begin{matrix} 
a & 0 & 0 \\
0 & b &  0 \\
0 & 0  & c \\ 
\end{matrix}\right) \in T
\] 
acts on $C_c(F^{\times})$ by right translation by $b$. The third bullet is clear. It remains to add normalizations and consider the restriction of $I(s)$. 
For the first bullet we use that the restriction of $|\det|$ (the character obtained by action on $V_2$) to $B$ is $\rho_U^{1/2}$. 
For the second bullet we use that the restriction of $|\det|$ on $V_2'$ is $|a|$.  We leave the details to the reader. We have the following:
\begin{itemize} 
\item $\Ind_B^{\SL_3}(\rho_U^{\frac{1}{2}(s+\frac12)})$;
\item $\Ind_B^{\SL_3}(\rho_U^{\frac{1}{4}(s+\frac12)} \cdot C_c(F^{\times}))$; 
\item $\ind_U^{\SL_3}(\psi_1)$ (Gelfand--Graev).
\end{itemize} 

Now we can compute the restriction of $\Theta(1_{H_a})$ to $\SL_3$. Of course, $\Theta(1_{H_a})=I(s)$, for all $a$,  however, the embedding of $\SL_3$ into 
$G_2$ depends on $a$ as explained: It is obtained by pre-conjugating the embedding for $a=1$ by $g\in \GL_3$ with $\det(g)=a$. Hence:

\begin{prop}  \label{P:filtration_theta_1} 
There exists an element $b\in F^{\times}$ such that for all $a\in F^{\times}$, $\Theta(1_{H_a})$ has an $\SL_3$-filtration with successive quotients 

\begin{itemize} 
\item $\Ind_B^{\SL_3}(\rho_U^{\frac{1}{2}(s+\frac12)})$;
\item $\Ind_B^{\SL_3}(\rho_U^{\frac{1}{4}(s+\frac12)} \cdot C_c(F^{\times}))$;
\item $\ind_U^{\SL_3}(\psi_{ab})$ (Gelfand--Graev).
\end{itemize} 
\end{prop} 
\begin{proof}  
By Proposition  \ref{P:conjugating_H}, the restriction of $\Theta(1_{H_a})$ to $\SL_3$ is obtained from the one for $a=1$ by pre-conjugating by an element of 
$\GL_3$ of determinant $a$. This explains the bottom piece. The top two are invariant under $\GL_3$-conjugation.
  \end{proof} 
The constant $b \in F^\times$ appears because the above computation (the restriction of $I(s)$ to $\SL_3$) uses a fixed embedding of $\SL_3$ to $G_2$, whereas Proposition \ref{P:filtration_theta_1} refers to a (possibly different) embedding resulting from the see-saw diagram. We shall argue, eventually, that $b$ is a cube; see Lemma \ref{L:bcube}.

 \section{Jacquet functors for $\SL_3$} \label{S:JMSL3}
In this section $d$ denotes 1, 2, and 4, for $E_6$, $E_7$, and $E_8$, respectively.
 
Let $P=MN$ be a maximal parabolic in $\SL_3$. We compute the normalized Jacquet functor $r_P(\Pi)$ in each of the three cases, by finding a maximal parabolic $\mathcal P$ in $E$ that intersects $\SL_3$ in $P$.  To that end, we realize $P$ so that $M$ is the centralizer of
\[ 
h=\left(
\begin{array}{ccc}  
-2& 0 & 0  \\ 
0 & 1 & 0  \\ 
0 & 0 & 1 
\end{array} \right) \in \mathfrak{sl}_3, 
\] 
while the Lie algebra of $N$ is the eigenspace of $h$, for the adjoint action, with eigenvalue $3$. The desired $\mathcal P$ (say in the simply connected group) comes similarly from $h$,  considered as an element of  
\[ 
\mathfrak e = \mathfrak{sl}_3 \oplus \mathfrak{g} \oplus  V_3  \otimes J 
  \oplus V_3^{\ast} \otimes  J^{\ast}. 
   \] 
Here, we identify $V_3=V_3^{\ast}$ with the spaces of column $3\times 1$ matrices, the action of $X\in  \mathfrak{sl}_3$ on 
  $V_3^{\ast}$ is by the matrix $-X^{\top}$, negative transpose of $X$.   
  
The Lie algebra of $\mathcal M$ is the centralizer of $h$. Note that 
$\mathcal M^{\rder}= \SL_2 \times G $. The unipotent radical $\mathcal N$ is described using the $\mathbb Z$-grading $\mathfrak e= \oplus_i \mathfrak e_i$ coming from the adjoint action of $h$ on $\mathfrak e$. 
To work out the grading, observe that 
\[ 
 V_3=  V_1^{\ast} \oplus V_2,
\] 
as a $M=\GL_2$-module, 
where $h$ acts on the two summands by eigenvalues -2  and 1, respectively.  Here $V_1^{\ast}$ is a line where $M=\GL_2$  acts by the determinant. 
The subspace $V_2\subset V_3$ is spanned by the vectors $e_2$ and $e_3$ of the standard basis.  
Thus the unipotent radical is a 3-step unipotent subgroup 
\[ 
\mathcal N=\mathcal N_1 \supset \mathcal N_2 \supset \mathcal N_3 \supset \mathcal N_4=0
\] 
such that $\mathcal N_i/\mathcal N_{i+1} \cong \mathfrak e_i$ for $1\leq i\leq 3$.  Each of these subquotients is naturally an $\mathcal M$-module, and here is their 
description, as $\SL_3 \times \SL_5$-modules:  
\[ 
\mathfrak e_1= V_2 \otimes J,  \, \mathfrak e_2= V_1 \otimes  J^{\ast},  \, \mathfrak e_3 = V_2. 
\] 
Observe that $\mathcal N_3=N$. The minimal representation has a $\mathcal P$-filtration 
\[ 
\Pi=\Pi_0 \supset \Pi_1 \supset \Pi_2 \supset \Pi_3 \supset \Pi_4=0 
\] 
such that $\Pi/\Pi_i = \Pi_{\mathcal N_i}$.  Since $\mathcal N_3=N$, the Jacquet module $r_P(\Pi)$ is equal to $\Pi/\Pi_3$ and thus has a filtration with three pieces.

The top is $\Pi/\Pi_1=\Pi_{\mathcal N}$, the Jacquet module with respect to $\mathcal P$. This is an explicit finite-length module for $\GL_2 \times G$. The center of $\GL_2$ acts on the subquotients  of  $r_P(\Pi/\Pi_1)= \rho_P^{-1/2} \cdot (\Pi/\Pi_1)$ (normalized Jacquet functor) by the characters 
\[ 
|\cdot|^{2d+1} \text{ and } |\cdot|^{3d}.
\]
Next, we have $\Pi_1/\Pi_2$. This is a module for the abelian group $\mathcal N/\mathcal N_2 \cong V_2 \otimes J$. The characters $\psi_x$ of this group are parameterized by elements $x$ in $V_2^{\ast} \otimes J^{\ast}$. By minimality of $\Pi$, the fiber
\[ 
(\Pi_1/\Pi_2)_{\mathcal N, \psi_x} 
\] 
is non-zero (in which case it is one-dimensional) if and only if $x$ is in the minimal $\GL_2 \times G$-orbit  $\Omega$ in $V_2^{\ast} \otimes J^{\ast}$. Thus 
\[ 
\Pi_1/\Pi_2 \cong C_c(\Omega) 
\] 
where $\GL_2 \times G$ acts geometrically. The orbit $\Omega$ consists of pure tensors $v\otimes z$ where $v$ is a non-zero element in $V_2^{\ast}$ and 
$z$ is a highest weight vector in $J^{\ast}$, viewed as a $G$ module.  We can pick $v=e_2$ so that the stabilizer, in $\GL_2$, 
of the line through $v$ is the group $B$ of lower-triangular matrices. 
Let $T$ be the torus of diagonal matrices in $B$.
We can identify $J$ and $J^{\ast}$ using the trace form. Using this identification, we can assume that 
\[ 
z=\left(
\begin{array}{ccc}  
1 & 0 & 0  \\ 
0 & 0 & 0  \\ 
0 & 0 & 0 
\end{array} \right) \in J. 
\] 
Let $\bar Q$ be the stabilizer in $G$ of the line through $z\in  J\cong J^*$.  
Let $Q$ be the stabilizer of the line through $z\in J$, for the standard action of $G$ on $J$. Then $L=Q\cap \bar Q$ is the Levi factor of these two parabolic subgroups. To describe $L$, we introduce the group 
\[
\almostGSpin_{2n} = \begin{cases}
    \mathbb{G}_m \times_{\mu_2} \Spin_{2n} \quad \text{(usual } \GSpin_{2n}\text{) } &\text{when $n$ is even};\\
    \mathbb{G}_m \times_{\mu_4} \Spin_{2n} &\text{when $n$ is odd};
\end{cases}
\]
see Appendix \ref{S:mini}. In the three cases that we consider, $L$ is
\[ 
\GL_2\times \GL_2 , \, (\GL_2\times \GL_4)^{\det}, \, \almostGSpin_{10}.
\]
Let $(B\times \bar Q)^{\det}$ be the subgroup fixing $v\otimes z$. 
Then $(B\times \bar Q)^{\det}$ acts on the one-dimensional space $(\Pi_1/\Pi_2 )_{\mathcal N, \psi_{v\otimes z}}$ by a character $\mu$. 
It is trivial on the unipotent subgroups of $B$ and $\bar Q$ since these are contained in 
the commutator subgroup. Hence $\mu$ is a character of $(T\times L)^{\det}$. For ease of exposition, assume for a moment that $\mu=1$. 
Observe that the $(T\times L)$-orbit of $v\otimes z$ is the punctured line $\ell$ through $v\otimes z$ with $0$ removed. Let 
\[ 
C_c(F^{\times}) := \ind_{(T\times L)^{\det}}^{T\times L}(1) \cong C_c(\ell).
\] 
The action of $T \times L$ on this piece gives an obvious correspondence of characters $\chi \leftrightarrow \chi$ of $T$ and $L$. 
Since $\GL_2 \times G$ acts transitively on $\Omega$, by induction in stages we have 
\[ 
\Pi_1/\Pi_2 \cong \ind_{B\times \bar Q}^{\GL_2 \times G} C_c(F^{\times}), 
\] 
where inductions are not normalized. After taking into account $\mu$ and using normalized parabolic induction, and normalized Jacquet functor, we have 
\[ 
r_P(\Pi_1/\Pi_2) \cong \Ind_{B\times \bar Q}^{\GL_2 \times G} \delta \cdot C_c(F^{\times}) 
\] 
where $\delta$ is a character of $T$ given by
\[ 
\left(
\begin{array}{ccc}  
a & 0  \\ 
0 & b 
\end{array} \right) \mapsto |a|^{d/2} \cdot |b|^d 
\] 
and the action of $T\times L$ on $C_c(F^{\times})$ is geometric, as explained before. At the end of this section we work this out in detail (and compute $\delta$) in the $E_6$ case, leaving the details of the $E_7$ and $E_8$ cases to the reader.

\vskip 5pt 

 The bottom term $\Pi_2/\Pi_3$ is a module for the two-step unipotent group $\mathcal N/\mathcal N_3$ with the center $\mathcal N_2/\mathcal N_3 \cong J^*$. 
The characters of $\psi_z$ of the center are parameterized by elements $z$ in $J$  but  the fiber  
\[ 
(\Pi_2/\Pi_3)_{\mathcal N_2, \psi_z}
\] 
is non-zero only for $z$ in the minimal orbit $\Omega \subset J$. In that case it
 is an irreducible representation of $\mathcal N/\mathcal N_3$ with the central character $\psi_z$, described as follows. We shall 
assume that $z$ is given by the above matrix.  
The commutator of the elements in $\mathcal N/\mathcal N_3$  gives an anti-symmetric bilinear form on $\mathcal N/\mathcal N_2 \cong V_2 \otimes J$ with values in 
$\mathcal N_2/\mathcal N_3  \cong J^{\ast}$ given by 
\[ 
[u\otimes x, v \otimes y ] = \langle u,v\rangle \cdot (x\times y) 
\] 
where $\langle u,v\rangle$ is the symplectic form on $V_2$ and $x\times y \in J^{\ast}$ is the cross product of $x,y\in J$, defined by
\[ 
(x\times y)(z)=(x,y,z) 
\] 
where $(x,y,z)$ is the trilinear form on $J$ invariant by $G$. Thus after we evaluate the commutator on $z$, the result is a symplectic form on $V_2\otimes J$. This form is degenerate since the bilinear form $(\cdot, \cdot, z)$ on $J$ is so. Since $L$ fixes the line through $z$, the kernel of this form is a subspace of $J$ invariant under $L$. 
Since $L$ is reductive, the kernel has an $L$-invariant complement $X \subset J$. It consists of elements 
\[ 
x=\left(
\begin{array}{ccc}  
0 & 0 & 0 \\
0 & b & w  \\ 
0 & \bar w & c 
\end{array} \right) \in J. 
\] 
The corresponding value  $(x,x,z)$ is (up to a non-zero constant) $bc-N(w)$, the determinant of the $2\times 2$ block.
Thus $\mathcal N$ has a quotient isomorphic to the Heisenberg group corresponding to the symplectic space $V_2 \otimes X$.  The Heisenberg representation, 
by inflation, gives an irreducible representation $\rho_z$ of $\mathcal N$  with the central character $\psi_z$.  We have 
\[ 
(\Pi_2/\Pi_3)_{\mathcal N_2, \psi_z}\cong \rho_z.
\] 
The stabilizer of $z$ in $\GL_2 \times L$ is $(\GL_2 \times L)^{\det}$. It acts on the above fiber.  The action is via the Weil representation.    
To that end, observe that $L$ is almost $\almostGSpin(X)$.  `Almost', because in the second case, $L/\SL_2\cong \GL_4 \cong \almostGSpin_6$. Thus on 
 \[ 
 W=\ind_{(\GL_2 \times L)^{\det}}^{\GL_2 \times L} (\rho_z) 
 \] 
 we have an action of $\GL_2 \times L$ which is essentially the classical theta correspondence for the groups of similitudes $\GL_2 \times \almostGSpin(X)$, where in the second case, the subgroup $\SL_2$ acts trivially. Putting things together, with normalized induction (see \cite{GS99} and \cite{BS24}), 
\[ 
r_P( \Pi_2/\Pi_3)\cong \Ind_Q^G (W). 
\]

 The following is a summary of the above discussion. 
 
 \begin{thm}   \label{T:three_step}
 The normalized Jacquet functor $r_P(\Pi)$, as a $\GL_2\times G$-module,  has a filtration with three subquotients, listed here from top to bottom, where $d=1,2,4$ in the three cases: 
 \begin{itemize} 
 \item A finite length $\GL_2\times G$-module, such that the center of $\GL_2$ acts on irreducible subquotients by the character $|\cdot|^{2d+1}$ or  $|\cdot|^{3d}$. In the cases $G=\SL_6$ and $E_6$, when $2d+1\neq 3d$, this is a direct sum of two irreducible $\GL_2\times G$-modules, isomorphic to the minimal and the trivial representation as $G$-modules.
 \item  $\Ind_{B\times \bar Q}^{\GL_2 \times G} \delta \cdot C_c(F^{\times})$, where $B$ is the group of lower triangular matrices, and $\delta$ 
 is a character of $T$ equal to $|\cdot|^{d/2}\otimes |\cdot|^d$. 
 \item  $\Ind_Q^G (W)$, where $W$ is the Weil representation of $\GL_2\times \almostGSpin_{2d+2}$.  
 \end{itemize} 
 
 \end{thm} 
 
 
 \begin{cor} Let $\pi$ be an irreducible representation of $G$. Then $r_P(\Theta(\pi))$ is a finite-length $\GL_2$-module.  
 \end{cor} 
 \begin{proof}  We have 
 \[ 
 r_P(\Theta(\pi))=(r_P(\Pi)\otimes \pi^{\vee})_G.
 \] 
 Now we use the filtration; for each of the three pieces the finiteness is clear. 
 \end{proof}

To end the section, we give some details in the $E_6$ case.  Let $V_3=F^3$ be the standard representation of $\SL_3$, realized as the space of column vectors. 
 Let $V_3^{\ast}=F^3$ be the dual representation: $g$ in $\SL_3$ acts on it by precomposing with transpose and inverse. Then we have ($S_3$ symmetry is visible) 
   \[ 
\mathfrak e = \mathfrak{sl}_3 \oplus \mathfrak{sl}_3 \oplus \mathfrak{sl}_3  \oplus V_3 \otimes V_3 \otimes V_3 
  \oplus V_3^{\ast} \otimes V_3^{\ast} \otimes V_3^{\ast}. 
\] 

Write $J=V_3\otimes V_3$ (the second two factors in the triple tensor product). 
Then $V_3\otimes V_3 \cong M_3$, the space of $3$ by $3$ matrices, via the isomorphism  $x\otimes y \mapsto x\cdot y^{\top}$.  
Then $(g,h)\in \SL_3 \times \SL_3$ acts on $X\in M_3$ by $gXh^{\top}$.  Similarly, we let  $J^{\ast} =V_3^{\ast}\otimes V_3^{\ast}$.  
Now $J^{\ast}\cong M_3$, but the action of $ \SL_3 \times \SL_3$ is precomposed by the transpose-inverse automorphism on 
each $\SL_3$ factor. 

We now discuss the geometry related to the Weil representation $W$ that appears at the bottom of the filtration. The stabilizer $Q$ in $G= \SL_3 \times \SL_3$ of the line through $z$ is $\bar P \times \bar P$.  The Levi factor $L$ of $Q$ 
is $M\times M = \GL_2 \times \GL_2$ and the orthogonal space $X\subset M_3$ is the space $M_2$ of $2\times 2$ matrices embedded in $M_3$ 
as a block.  Thus $X\cong V_2 \otimes V_2$.  

Observe that 
\[ 
\alpha \times 
\left( \begin{array}{cc}  
\det(\beta)^{-1} & 0   \\
0 & \beta
\end{array} \right) 
 \times 
\left( \begin{array}{cc}  
\det(\gamma)^{-1} & 0   \\
0 & \gamma
\end{array} \right) \in \GL_2 \times M \times M 
\] 
acts on $e_1\otimes z\in V_1^{\ast} \otimes J$ 
 by $\det(\alpha\gamma\beta)^{-1}$. Hence 
 \[ 
 (\GL_2 \times L)^{\det}= (\GL_2 \times \GL_2 \times \GL_2)^{\det}. 
 \] 
 This group acts naturally on the symplectic space $V_2 \otimes V_2 \otimes V_2$. 
The similitude group correspondence on $W$ works as follows \cite{GS23,HS25}: if $\tau$ is an irreducible Whittaker-generic representation of $\GL_2$, then
\[ 
\Theta(\tau) =\tau\otimes \tau  \text{ and } \Theta(\tau\otimes \tau)= \tau. 
\]
It remains to discuss the middle part of the filtration, namely $\Omega \subset V_2^{\ast} \otimes J^*$. The stabilizer of the line $z$, viewed as an element of $J^{\ast}$, is $\bar Q= P  \times P$. The stabilizer of $e_2\otimes z\in V_2^{\ast} \otimes J^{\ast}$ in $ T \times M \times M$ consists of all elements 
\[ 
\left( \begin{array}{cc}  
a & 0   \\
0 & b
\end{array} \right) 
 \times 
\left( \begin{array}{cc}  
\det(\beta)^{-1} & 0   \\
0 & \beta
\end{array} \right) 
 \times 
\left( \begin{array}{cc}  
\det(\gamma)^{-1} & 0   \\
0 & \gamma
\end{array} \right) \in T \times M \times M 
\] 
and acts on $e_2\otimes z\in V_2^{\ast} \otimes J^{\ast}$  by $a^{-1} \det(\gamma\beta)$. Hence in this case the action 
of $T \times M \times M$ on $f\in C_c(F^{\times})$ is by left translation by $a^{-1}$ and right translation by $ \det(\gamma\beta)$,  
with a twist by a character on $T$ given by $|\cdot |^s \times |\cdot |^t$ for two real numbers that we need to compute.  Now, by the 
Frobenius reciprocity, for a generic character $\chi$ of $F^{\times}$, the bottom piece of the filtration implies that 
\[ 
\theta(\Ind_P(\chi))= \Ind_{\bar P}(\chi) \otimes \Ind_{\bar P}(\chi) \cong  \Ind_{P}(\chi) \otimes \Ind_{P}(\chi). 
\] 
Next, we use that $r_P(\Ind_P(\chi))$ is a direct sum of $\tau=|\cdot |^{1/2}\chi^{-1}\times |\cdot |$ and a one-dimensional $\GL_2$-module. 
$\chi\circ \det$. Thus both 
\[ 
\tau \otimes \Ind_{P}(\chi) \otimes \Ind_{P}(\chi) \text{ and }   (\chi\circ \det) \otimes \Ind_{P}(\chi) \otimes \Ind_{P}(\chi) 
\] 
are quotients of $r_P(\Pi)$, where the former can only be the quotient of the middle term in the filtration. 
This forces $s=1/2$ and $t=1$.

In the above we used that, if $\tau = \chi\times \mu$, then $r_P(\Ind_P(\tau))$ has a composition series consisting of 
\[  
\chi\times \mu , \,  \chi^{-1} \times (\mu/\chi), \, (\chi/\mu) \times \mu^{-1}.  
\]

\section{Lifts for non-cuspidal representations}\label{S:noncusp}

In this section we derive several consequences of Theorem \ref{T:three_step},  and all but completely determine the correspondence for non-cuspidal representations of $\SL_3$. We retain the notation from the previous section; in particular, $P=MN$ is a maximal parabolic in $\SL_3$, and $Q$ (with Levi $L$) is the parabolic in $G$ that appears in Theorem \ref{T:three_step}.

\subsection{Finiteness of lift to $\SL_3$}  

\begin{prop} Let $\pi$ be an irreducible representation of $G$. Then $\Theta(\pi)$ has finite length. 
\end{prop} 
\begin{proof} 
Fix a Whittaker datum $(U,\psi)$ for $\SL_3$, we have the see-saw identity 
\[ 
\Hom_U(\Theta(\pi), \psi) \cong \Hom_G(C_c(G/H), \pi). 
\] 
Since $H$ is a symmetric group, $\Hom_G(C_c(G/H), \pi)$ is finite-dimensional; see \cite{De10} and \cite{SV17}. Hence $\Theta(\pi)_{U,\psi}$ is 
finite-dimensional for any choice of the Whittaker datum $(U,\psi)$. 

This takes care of the supercuspidal part of $\Theta(\pi)$. To finish, we need to show that $r_P(\Theta(\pi))$ is a finite-length $M$-module for any maximal parabolic $P=MN$ of $\SL_3$. To that end, 
\[ 
r_P(\Theta(\pi))=(r_P(\Pi)\otimes \pi^{\vee})_G.
\] 
Now we use the filtration of $r_P(\Pi)$ from Theorem \ref{T:three_step}.  For each of the three pieces the finiteness is clear. 
\end{proof} 

\subsection{Lifting principal series representations of $\SL_3$}   
\begin{prop} \label{P:theta_nc} 
Let $\tau$ be an irreducible representation of $\GL_2$ such that the center of $\GL_2$ does not act as $|\cdot|^{2d+1}$ or $|\cdot|^{3d}$, and $r_{\bar B}(\tau)$ does not contain a character $\chi\otimes |\cdot |^d$. Let $\Theta_{\mini}(\tau)$ be the mini-theta lift to $L$; see Appendix \ref{S:mini}.  Then 
\[  
\Theta(\Ind_P^{\SL_3}(\tau))= \Ind_Q^G(\Theta_{\mini}(\tau)). 
\] 
More generally, if $\sigma$ is a submodule of $\Ind_P^{\SL_3}(\tau)$, then $\Theta(\sigma)$ is a quotient of $\Ind_Q^G(\Theta_{\mini}(\tau))$. 
\end{prop} 
\begin{proof}  By Frobenius reciprocity,  
  \[ 
  (\Pi\otimes \Ind_P^{\SL_3}(\tau)^{\vee})_{\SL_3} \cong (r_P(\Pi)\otimes \tau^{\vee})_{\GL_2}. 
  \] 
  Now, the conditions on $\tau$ ensure that we avoid the top two pieces in the filtration of $r_{P}(\Pi)$.
  \end{proof} 

  Let us list cases where the conditions of the above proposition are satisfied: 
  \begin{itemize} 
  \item $\tau=\chi\otimes |\det|^s$ where $s\leq 0$ and $\chi$ is a unitary character of $\GL_2$;
  \item $\tau=\text{St}_{\chi}\otimes |\det|^s$ where $s\leq 0$, and $\text{St}_{\chi}$ is the Steinberg representation twisted by the unitary character $\chi$;
  \item $\tau$ is a principal series induced from $\chi\times \mu$ where $|\chi|=|\cdot |^a$ and $|\mu|=|\cdot|^b$ where $a,b\leq 0$; 
  \item $\tau$ is a supercuspidal representation;
  \item $\tau$ is a unitarizable representation. 
  \end{itemize} 

  For example, let us verify the condition for $\tau$ in the first two bullets. We set $\chi=1$ to simplify notation. If $\tau=|\det|^s$ then its central character is $|\cdot|^{2s}$ and $r_{\bar B} (\tau)  = |\cdot |^{-1/2+s} \otimes |\cdot|^{1/2+s}$. 
 If $\tau=|\det|^{s}\otimes \text{St}$ then its central character is $|\cdot|^{2s}$ and $r_{\bar B} (\tau)  = |\cdot |^{1/2+s} \otimes |\cdot|^{-1/2-s}$. 
Since $s\leq 0$ the conditions of Proposition \ref{P:theta_nc} are satisfied.

Any non-square-integrable $\sigma$ can be embedded into 
$\Ind_P^{\SL_3}(\tau)$ with $\tau$ satisfying one of the first three bullets, or in $\Ind_{\bar P}^{\SL_3}(\tau)$ with $\tau$ satisfying opposite conditions, namely, $s\leq 0$ and $a,b\leq 0$.  Now, there is a sister 
statement to Theorem \ref{T:three_step} which involves $r_{\bar P}(\Pi)$. 
Thus $\Theta(\sigma)$ is, respectively a quotient of 
\[ 
\Ind_Q^G(\Theta_{\mini}(\tau)) \text{ or } \Ind_{\bar Q}^G(\Theta_{\mini}(\tau)).
\]
If $\sigma$ is non-tempered, then it is equal to the induced representation, with three exceptions: the trivial representation, and two degenerate representations with the same infinitesimal character. We shall analyze the lift of the trivial representation in a moment; before that, we record the following two easy consequences that we shall need later.


   \vskip 5pt 
\begin{cor} \label{P:small_e6}  Let $\sigma$ be an irreducible and degenerate representation of $\SL_3$. Then the small theta lift to $\SL_3\times\SL_3$ is 
$\theta(\sigma)\cong \sigma\otimes\sigma$.   
\end{cor} 
\begin{proof} Without loss of generality we can assume that $\sigma$ is a unique quotient of $\Ind_{\bar P}^{\SL_3}(\tau)$ where $\tau$ is either a one-dimensional representation $\chi\otimes |\det|^s$ with $s\leq 0$ and $\chi$ a unitary character, or $\tau=|\det|^{-3/2}\otimes \St$ where $\St$ is the 
Steinberg representation of $M=\GL_2$.  Hence $\sigma$ is a submodule of $\Ind_{P}^{\SL_3}(\tau)$. Since 
 the conditions of Proposition \ref{P:theta_nc} are satisfied (here $d=1$), $\Theta(\sigma)$ arises from the bottom term of the filtration. This implies that 
 $\theta(\sigma)$ is a quotient of 
\[ 
\Ind_{\bar P}^{\SL_3}(\tau) \otimes \Ind_{\bar P}^{\SL_3}(\tau), 
\] 
 as desired.  
\end{proof}


   \begin{cor}\label{C:unitary} Let $\tau$ be an irreducible unitary representation of $\GL_2$. Let $\pi$ be a summand of $\Ind_P^{\SL_3}(\tau)$. Then $\Ind_Q^G(\Theta_{\mini}(\tau))$ is unitary and 
   contains $\Theta(\pi)$ as a summand.  
   \end{cor} 
   \begin{proof} It suffices to observe that $\Theta_{\mini}(\tau)$ is unitary since $\tau$ is so.  
   \end{proof} 

\subsection{Lifting the trivial representation of $\SL_3$} 
The trivial representation lives in a principal series representation with the Loewy diagram 

\[ 
\begin{array}{ccc}  
 &  1 &  \\
\sigma_2 &  & \sigma'_2  \\ 
 & \St &  
\end{array} 
\] 
where the subscript 2 denotes the dimension of the space of Iwahori-fixed vectors. In particular, the socle is the Steinberg representation and the co-socle is the trivial representation.

Assume $G=E_6$. The principal series representations of $\SL_3$ lift to representations of $G$ appearing in a degenerate principal series representation $J(s,t)$ induced from the co-rank two Levi $D_4$. Here $J(s,t)$ denotes induction from the character $|\varpi_1|^s\cdot |\varpi_6|^t$. For the lift of the trivial representation, we are interested in $J(1,1)$. This is the `IM-involute' of $X(D_4,1)$ from \cite[p.~456]{Ci06}. Its subquotients are described by the following Loewy diagram 
\[ 
\begin{array}{ccc}  
\pi_{57} &  & \pi_{21}  \\ 
\pi_{84} &  & \pi'_{84}  \\ 
 & \pi_{24} &  
\end{array}
\] 
(Here again the subscript denotes the dimension of the space of Iwahori-fixed vectors.) The bottom is the socle and the top is the co-socle. The dual of $J(1,1)$ is $J(-1,1)$ which has the same irreducible subquotients, but the diagram is upside down. 

\begin{prop}\label{P:triviallifts} We have $\theta(\sigma_2)=\pi_{84}$,  $\theta(\sigma'_2)=\pi'_{84}$, and $\Theta(1)=\pi_{21}\oplus \pi_{57}$.
\end{prop}
\begin{proof} We use that $\sigma_2$ is a submodule of $\Ind_{P}^{\SL_3} (\St\otimes |\det|^{-3/2})$ so that $\Theta(\sigma_2)$ is a 
quotient of 
\[ 
\Ind_Q^G(\Theta_{\mathrm{mini}}(\St\otimes |\det|^{-3/2})). 
\] 
This representation is a submodule of $J(1,1)$ of length two, with 
$\pi_{24}$ as a submodule, and $\pi_{84}$ as a quotient. This proves that 
$\theta(\sigma_2)=\pi_{84}$.  The second statement is proved similarly. Now for the lift of $1$, we use the sequence 
\[ 
0\rightarrow 1 \rightarrow \Ind_P^{\SL_3}(|\det|^{-3/2})\rightarrow \sigma'_2 \rightarrow 0.
\]
Applying the theta correspondence, this becomes
\[ 
\Theta(\sigma'_2) \rightarrow \Ind_Q^G(\Theta_{\mathrm{mini}}(|\det|^{-3/2})) \rightarrow \Theta(1) \rightarrow 0.  
\]
The middle term is a quotient of $J(1,1)$ of length three, with $\pi'_{84}$ as a submodule, and $\pi_{21}\oplus \pi_{57}$ as a quotient. We know that 
$\Theta(\sigma'_2)$ has $\pi'_{84}$ as a quotient, but we do not have control of the left arrow. This leaves two possibilities: $\Theta(1)$ as claimed, or slightly larger, with $\pi'_{84}$ as a submodule. However, we can repeat the same argument by embedding $1$ into a representation induced from $\bar P$, with the quotient $\sigma_2$.  Combining the two gives us the claimed statement.   
\end{proof}

\subsection{Lifting Steinberg}
We begin by describing some standard modules that we need.

Let $X(A_2\times A_2, \tau, s)$ be the standard module for $\SL_6$ induced from the standard parabolic with Levi $A_2\times A_2$: we are inducing a tempered representation $\tau$ twisted by a positive character of the Levi, whose logarithm is 
\[ 
s\varpi_3=\frac{1}{2}(s,s,s,-s,-s,-s)
\]
where $\varpi_3$ is the third fundamental weight for $\SL_6$. 
Of interest to us is $X(A_2\times A_2, \St, 1)$ where 
$\St$ is the Steinberg representation. The unique (Langlands) 
quotient $\bar X(A_2\times A_2, \St, 1)$ is a Speh representation of $\SL_6$.

Similarly, let $X(A_2\times A_2, \tau, (s,t))$ be the standard module for $E_6$, induced from the standard parabolic with Levi $A_2\times A_2$. Here we are inducing a tempered representation $\tau$ twisted by a positive character of the Levi whose logarithm is 
\[ 
s\varpi_4+ t\varphi_2=(0,0,s,s,s,-s,-s,s) + 
\frac{1}{2}(t,t,t,t,t,-t,-t,t)
\]
where $\varpi_2$ and $\varpi_4$ are fundamental weights for $E_6$.
Of interest to us is the Langlands quotient 
$\bar X(A_2\times A_2, \St, (1,1))$. This representation has been studied by Ciubotaru \cite[page 462]{Ci06}. It has a single 
$W$-type of dimension 24. In fact, this is none other than $\pi_{24}$, the socle of $J(1,1)$. We have: 

\begin{prop} Let $\St$ be the Steinberg representation of $\SL_3$. 
\begin{enumerate}[(i)]
    \item The big theta of $\St$ lift to $\SL_6$ is irreducible and isomorphic to $\bar X(A_2\times A_2, \St, 1)$. 
    \item The big theta lift of $\St$ to $E_6$ is irreducible and isomorphic to $\pi_{24}=\bar X(A_2\times A_2, \St, (1,1))$.
\end{enumerate}
\end{prop} 
\begin{proof} We shall use Theorem \ref{T:three_step}. Recall that $d=2$ and $4$ in the two cases that we are considering. 
Since $\St\otimes \Theta(\St)$ is a quotient of $\Pi$, we have a surjective map 
\[ 
r_P(\Pi) \rightarrow r_P(\St)\otimes \Theta(\St).
\]
Let $\tau:=r_P(\St)$. Note that $\tau\cong |\det|^{3/2}\otimes \St$. It follows that the central character of $\tau$ 
is $|\cdot |^3$. Thus $\tau$ cannot be a quotient of the first piece of the filtration of $r_P(\Pi)$. For the second, note that, since $\bar B$ are upper triangular matrices, 
$r_{\bar B}(\tau)= |\cdot|^{1/2 + 3/2} \times  |\cdot|^{-1/2 + 3/2}$. In particular, the second factor is not $|\cdot|^d$. Hence 
$\tau$ is not a quotient of the middle piece, and $\Theta(\St)$ comes from the bottom piece. As a result, $\Theta(\St)$ is a quotient of 
\[
\Ind_Q^G(\Theta_{\mini}(\tau)).
\]
Assume $G=E_6$, so $Q$ is a $D_5$ parabolic in $E_6$. 
The above induced representation is a quotient of $J(-1,-1)$, of length two. 
It has $\pi_{84}$ as a submodule and $\pi_{24}$ as a quotient. 
 
Now we play the same game with $\bar P$. 
This time we conclude that $\Theta(\St)$ is a quotient of a representation induced from  $\bar Q$ parabolic of length two, with 
$\pi'_{84}$ as a submodule and $\pi_{24}$ as a quotient. 
Thus $\Theta(\St)=\pi_{24}$. 
\end{proof}

\subsection{Lifting trivial and minimal representations of $G$}
Here $G=\SL_6$ and $E_6$.  
Let $1_G$ be the trivial representation of $G$. If $G=E_6$ let 
$\pi_{\min}$ be the unique minimal representation of $G$. 
On the other hand, the group $\SL_6$ has a family of minimal representations. They are induced from an $\SL_5$ maximal parabolic. We shall be interested in two, so let $\pi_{\min}$ denote a minimal representation of $\SL_6$ with infinitesimal character 
\[
(13/6,7/6,1/6,-5/6,-5/6, -11/6) \text{ or } (11/6,5/6,5/6,-1/6,-7/6, -13/6).
\]
In particular, these are not hermitian.

Recall (from Theorem \ref{T:three_step}) that the top of the filtration of $r_P(\Pi)$ as $(\GL_2\times G)$-module is a sum 
\[ 
|\det|^{d+ \frac{1}{2}}\otimes \pi_{\min}\ \oplus\ \tau \otimes 1_G
\]
where $\tau$ is induced from $|\cdot|^d \times |\cdot |^{2d}$.  If $G=\SL_6$, then $\pi_{\min}$ depends on the conjugacy class of $P$. 

\begin{prop} \label{P:lift_min}
\begin{enumerate}[(i)]
    \item The big theta lift of $1_G$ has an irreducible quotient isomorphic to the irreducible principal series representation with the infinitesimal 
character $(d,0,-d)$.
    \item If $G=E_6$, the big theta lift of $\pi_{\min}$ has 
a quotient isomorphic to a sum of two non-hermitian degenerate representations of infinitesimal characters 
\[ 
(2, 1, -3) \text{ and } (-2 , -1, 3).
\] 
\item If $G=\SL_6$ the big theta lifts of the two minimal representations have quotients (in some order) isomorphic to the degenerate representations with infinitesimal characters 
\[ 
\frac{1}{3}(4, 1, -5) \text{ and } \frac{1}{3}(-4 , -1, 5). 
\]
\end{enumerate}
\end{prop}
\begin{proof}
This follows immediately from Frobenius reciprocity. The full lift of 
$1_G$ has a quotient $\Ind_P^{\SL_3}(\tau)$ which is an irreducible principal series with infinitesimal character $(d,0,-d)$. Similarly, the full lift of $\pi_{\min}$ has  a quotient isomorphic to the degenerate representation 
$\Ind_P^{\SL_3}(|\det|^{d+ \frac{1}{2}})$. This is not a hermitian representation. Its hermitian dual is also a quotient of $\Theta(\pi_{\min})$ because we can use the same argument with the top quotient of $r_{\bar P}(\Pi)$, where $\bar P$ is the opposite parabolic. 
\end{proof}
 
\section{Three $\SL_3$}  \label{S:3SL3}
In this section we work out the correspondence inside $E_6$. This resembles the results of \cite{HS25}. We have 
 \[ 
\SL_3\times \SL_3\times \SL_3 /\mu_3^2 
 \] 
 in adjoint $E_6$.  Let $\Pi$ be the minimal representation of $E_6$. 
 The goal here is to prove that for any Whittaker-generic irreducible representation $\pi$ of $\GL_3$,  
 \[ 
\Theta(\pi):= (\Pi\otimes \pi^{\vee})_{\SL_3}\cong \pi\otimes\pi  
 \] 
where the former is naturally a $\GL_3^2$-module, following our paper \cite{BGS24}. We start with the following observation: On $\SL_n$ we have an action  of $(g,h)\in (\GL_3\times \GL_3)^{\det}$  by $x\mapsto gxh^{\top}$. This gives an action of $(\GL_3\times \GL_3)^{\det}$ on the regular representation $C_c(\SL_n)$. 
 \begin{prop} \label{P:two_sl}
 For every irreducible representation $\pi$ of $\GL_n$ 
  \[ 
  (C_c(\SL_n) \otimes \pi^{\vee})_{\SL_n} \cong \pi 
  \] 
  where we are taking co-invariants for $\SL_n$ acting on $C_c(\SL_n)$ from the right.  
  \end{prop} 
  \begin{proof} We use that $\pi^{\vee}$ is realized on the same space $V$ as $\pi$, but $\pi^{\vee}(g^{-1}) =\pi(g^{\top})$.  
  We have a natural $\GL_n$-intertwining map 
  \[ 
  (C_c(\SL_n) \otimes \pi^{\vee})_{\SL_n} \rightarrow \pi 
  \] 
  given by $(f,v) \mapsto \pi(f)v$, where $v\in V$.  
  To prove that this is bijective, let $\pi=\pi_1 + \cdots + \pi_s$ be the restriction to $\SL_n$, and use 
  \[ 
  (C_c(\SL_n) \otimes \pi^{\vee}_i)_{\SL_n} \cong  \pi_i
  \] 
  for any irreducible representation $\pi_i$ of $\SL_n$.  
   \end{proof} 
   
Let us look at a special case when an irreducible representation $\pi$ of $\GL_n$ stays irreducible when restricted to $\SL_n$. Let $\pi_1$ denote the restriction to $\SL_3$. Then $\pi\otimes\pi$ is a representation of $(\GL_3\times \GL_3)^{\det}$ that stays irreducible on $\SL_3\times\SL_3$. Furthermore, $\pi\otimes\pi$ has a unique 
   $\Delta(\GL_n)$-invariant functional. Since 
   \[ 
   \SL_n= (\GL_3\times \GL_3)^{\det}/\Delta(\GL_n), 
   \] 
$\pi\otimes\pi$ is an $(\GL_3\times \GL_3)^{\det}$-invariant quotient of $C_c(\SL_n)$. 
Let $\pi'$ be another irreducible representation $\pi$ of $\GL_n$ whose restriction to $\SL_n$ is $\pi_1$. Then 
$\pi'\otimes\pi'$ is also a quotient of $C_c(\SL_n)$. Since $\pi_1\otimes\pi_1$ appears as a quotient of $C_c(\SL_n)$ with multiplicity one, it follows that 
$\pi\otimes\pi\cong \pi'\otimes\pi'$, as $ (\GL_3\times \GL_3)^{\det}$-modules.

  \begin{prop} \label{P:gl3_correspondence}Let $\pi$ be an irreducible Whittaker-generic representation of $\GL_3$. 
   Then $\Theta(\pi)\cong \pi\otimes\pi $, and $\Theta(\pi\otimes\pi)=\pi$.  
   \end{prop} 
  \begin{proof}  Let $(U,\psi)$ be a Whittaker datum for one of the two $\GL_3$'s acting on $\Theta(\pi)$. Then
  \[ 
   \Pi_{U, \psi}=C_c(\SL_3),
  \] 
  the regular representation of $\SL_3$. By Proposition \ref{P:two_sl}, $\Theta(\pi)_{U,\psi}\cong \pi$, for the other $\GL_3$. We can repeat 
 this argument, switching the roles of the two $\GL_3$. It follows that $\pi\otimes \pi$ is a subquotient of $\Theta(\pi)$ where all other irreducible subquotients $\pi_1\otimes\pi_2$ involve Whittaker-degenerate representations $\pi_1$ and $\pi_2$ of $\GL_3$. 
 By the K\"uneth formula, 
 \[ 
 \Ext^1(\pi_1\otimes\pi_2, \pi\otimes\pi) \cong  \Hom(\pi_1, \pi) \otimes  \Ext^1(\pi_2, \pi) 
 \oplus  \Ext^1(\pi_1, \pi) \otimes  \Hom(\pi_2, \pi) =0. 
 \] 
 Hence there exists a degenerate quotient $\pi_1\otimes\pi_2$ of $\Theta(\pi)$.  This in turn implies that $\pi_1\otimes\pi_2\otimes \pi$ is an 
$\SL_3^3$-invariant quotient,  contradicting $\theta(\pi_1)\cong\pi_1\otimes\pi_1$, which is guaranteed by Proposition \ref{P:small_e6}. 

The other direction follows from the first.  
\end{proof} 
 
We can now formulate the main result about the correspondence in the adjoint $E_6$, where $\SL_3$ forms a dual pair with the group $(\GL_3\times \GL_3)^{\det}/\Delta \GL_1$. 

\begin{thm}  Let $\pi$ be an irreducible representation of $\GL_3$. The small theta correspondence gives a bijection between irreducible 
$\SL_3$-summands of $\pi$ and irreducible $(\GL_3\times \GL_3)^{\det}$-summands of $\pi\otimes\pi$.  
\end{thm} 
 \begin{proof} If $\pi$ is Whittaker-generic, this follows from Proposition \ref{P:gl3_correspondence} and \cite{BGS24}.  
 So assume that $\pi$ is a degenerate representation. Then the restriction to 
 $\SL_3$ is irreducible. Denote it $\pi_1$.   
 By Proposition \ref{P:small_e6}, we have $\theta(\pi_1)\cong \pi_1\otimes\pi_1$, as $\SL_3^2$-modules. We now show that, as a $(\GL_3\times \GL_3)^{\det}$-module, $\theta(\pi_1)$ is the 
 restriction of $\pi\otimes\pi$.  It suffices to show that $\theta(\pi\otimes\pi)\neq 0$, for it will have to be isomorphic to $\pi_1$. 
  Let $(U,\psi)$ be a Whittaker datum for $\SL_3$, so that 
  \[ 
   \Pi_{U, \psi}=C_c(\SL_3),
  \] 
  the regular representation of $\SL_3$. We know that $\pi\otimes\pi$ is a quotient of $C_c(\SL_3)$, and therefore $\theta(\pi\otimes\pi)\neq 0$. 
 \end{proof}

 \begin{prop}  Let $a\in F^{\times}$.  Let $\pi$ be an irreducible representation of $\GL_3$. There is a unique irreducible $(\GL_3\times \GL_3)^{\det}$-summand $\sigma$ of $\pi\otimes\pi$ such that 
 $\Theta(\sigma)_{U,\psi_a}$ is non-zero; for this summand, the coinvariant space is one-dimensional.  
 \end{prop} 
 \begin{proof} If $\pi$ is Whittaker-generic, let $\pi_1, \ldots, \pi_k$ be its irreducible $\SL_3$-summands. There is only one which is Whittaker-generic for $\psi_a$. Recall that  $\sigma_1=\Theta(\pi_1), \ldots,  \sigma_k=\Theta(\pi_k)$ are precisely the irreducible $(\GL_3\times \GL_3)^{\det}$-summands of $\pi\otimes\pi$. Since $\pi_i=\Theta(\sigma_i)$, we are done in this case.  Assume now that $\pi$ is degenerate and hence that $\pi\otimes\pi$ is an irreducible representation of $(\GL_3\times \GL_3)^{\det}$. 
 Then, by \S \ref{subs_Whittaker_coinv},
 \[ 
  \Theta(\pi\otimes\pi)_{U,\psi_a} \cong \Hom_{(\GL_3\times \GL_3)^{\det}}(C_c(\omega_a), \pi \otimes \pi)\cong \mathbb C,
  \] 
  which finishes the proof.
  \end{proof} 

\begin{rmk}
We take this opportunity to correct an error in \cite[p.~199, last paragraph]{BGS24} which was pointed out to us by Yi Shan.  Namely, using the notation and terminology of \cite{BGS24}, it was asserted there that 
\[ \tilde{J}^{sim} \cap (T(F) \times T(F)) = T^{\nabla}(F).  \]
This should read:
\[ \tilde{J}^{sim} \cap (T(F) \times T(F)) =  (Z(F) \times Z(F)) \cdot T^{\nabla}(F).\]
As such, one cannot extend $\Omega$ to the intermediate group 
\[  \tilde{J}^{sim}(F) \cdot (T(F) \times T(F)). \]
Instead, one needs to consider the representation  $\pi_{Z, \chi_Z}$ which is the $(Z, \chi|_Z)$-isotypic quotient of $\Omega$ (where $Z$ is the central subgroup of $G$ or $H$). One can then extend $\Omega_{Z, \chi|_Z}$ to the intermediate group above by letting $T(F) \times T(F)$ act by $\chi \otimes \chi$. Denoting this extended module by $\Omega_{\chi}$, one can then define $\tilde{\Omega}_{\chi}$  by induction, as in \cite[p.~199, last line]{BGS24}. 
 This does not affect the rest of \textnormal{[op.cit.]}.
 \end{rmk} 
  
   
   \section{The cuspidal support of theta lifts}  \label{S:support}
Let $\sigma$ be a supercuspidal representation of $\SL_3$. We narrow down the cuspidal support of $\Theta(\sigma)$ for the correspondences in $E_7$ and $E_8$. 

\begin{prop} \label{P:support_1} 
 Let $\sigma$ be a supercuspidal representation of $\SL_3$. Let $\Theta(\sigma)$ be the lift to $\SL_6$. 
Let $P=MN$ be a maximal parabolic subgroup of $\SL_6$ such that $M^{\rder}$ is not $\SL_3 \times \SL_3$. Then $\Theta(\sigma)_N=0$. 
\end{prop} 
\begin{proof} We need to prove that $\Pi_N$ contains no supercuspidal representations of $\SL_3$. To that end, 
 we shall construct a maximal parabolic subgroup $\mathcal P=\mathcal M\mathcal N$ in the ambient group that intersects $\SL_3$ in $P$ 
and describe the restriction of $\Pi$ to $\mathcal P$.   

Let $V_n$ denote the standard $n$-dimensional representation of $\SL_n$. 
Recall that 
\[ 
\mathfrak e = \mathfrak{sl}_3 \oplus \mathfrak{sl}_6 \oplus V_3 \otimes \wedge^2 V_6  \oplus V_3^{\ast} \otimes \wedge^2 V_6^{\ast}. 
\] 

Assume first that $M^{\rder} = \SL_5$.  We can realize $P=MN$ so that $M$ is the centralizer of
\[ 
h=\frac{1}{2}\left(
\begin{array}{cccccc}  
1 & 0 & 0 & 0 & 0 & 0 \\ 
0 & 1 & 0 & 0 & 0 & 0 \\ 
0 & 0 & 1 & 0 & 0 & 0 \\ 
0 & 0 & 0 & 1 & 0 & 0 \\ 
0 & 0 & 0 & 0 & 1 & 0 \\ 
0 & 0 & 0 & 0 & 0 & -5 
\end{array} \right) \in \mathfrak{sl}_6, 
\] 
while the Lie algebra of $N$ is the eigenspace of $h$, for the adjoint action, with eigenvalue $3$. The desired $\mathcal P$ comes 
similarly from $h$,  considered as an element of $\mathfrak{sl}_6\subset \mathfrak e$. More precisely, $\mathcal M$ is the centralizer of $h$. Note that 
$\mathcal M^{\rder}= \SL_3 \times \SL_5$. The unipotent radical $\mathcal N$ is described using the 
$\mathbb Z$-grading $\mathfrak e= \oplus_i \mathfrak e_i$ coming from the adjoint action of $h$ on $\mathfrak e$. 
To work out the grading, observe that 
\[ 
\wedge^2 V_6= \wedge ^2 V_5  \oplus V_5,
\] 
as an $\SL_5$-module, 
where $h$ acts on the two summands by eigenvalues $1$ and $-2$, respectively.  
Thus the unipotent radical is a $3$-step unipotent subgroup 
\[ 
\mathcal N=\mathcal N_1 \supset \mathcal N_2 \supset \mathcal N_3 \supset \mathcal N_4=0
\] 
such that $\mathcal N_i/\mathcal N_{i+1} \cong \mathfrak e_i$ for $1\leq i\leq 3$.  Each of these subquotients is naturally an $\mathcal M$-module. As $\SL_3 \times \SL_5$-modules, they are:
\[ 
\mathfrak e_1= V_3 \otimes \wedge^2V_5,  \, \mathfrak e_2= V_3^{\ast} \otimes V_5^{\ast},  \, \mathfrak e_3 = \mathbb C \otimes V_5. 
\] 
Observe that $\mathcal N_3=N$. The minimal representation has a $\mathcal P$-filtration 
\[ 
\Pi=\Pi_0 \supset \Pi_1 \supset \Pi_2 \supset \Pi_3 \supset \Pi_4=0 
\] 
such that $\Pi/\Pi_i = \Pi_{\mathcal N_i}$.  Now each quotient $\Pi_i/\Pi_{i+1}$ is a $\mathcal P$-module, so it suffices to show that
$(\Pi_i/\Pi_{i+1})_N$ nas no supercuspidal $\SL_3$ subquotients for $0\leq i \leq3$.  We work from the bottom. 
Since $N=\mathcal N_3$, we have $(\Pi_3/\Pi_4)_N=0$, and there is nothing to prove here. Now assume $0< i<3$. 
Observe that $\mathcal N_{i+1}$ acts trivially on $\Pi_i/\Pi_{i+1}$, hence 
$\Pi_i/\Pi_{i+1}$ is a module for the abelian group 
$\mathcal N_i/\mathcal N_{i+1}$.  Characters $\psi_x$ of this group are naturally parameterized by elements $x\in \bar{\mathcal N}_i/\bar{\mathcal N}_{i+1}\cong \mathfrak e_{-i}$  
where $\bar{\mathcal N}_i$ is the filtration of the opposite unipotent subgroup $\bar{\mathcal N}$.  Thus, if we abbreviate $V= \Pi_i/\Pi_{i+1}$, then $V$ defines a sheaf $\mathcal V$ on 
$\bar{\mathcal N}_i/\bar{\mathcal N}_{i+1}$ whose stalks are 
\[ 
\mathcal V_x= V_{\mathcal N_i/\mathcal N_{i+1}, \psi_x}. 
\] 
Since $\Pi$ is minimal, this sheaf is supported on the minimal $\mathcal M$-orbit $\Omega$ in $\bar{\mathcal N}_i/\bar{\mathcal N}_{i+1}$. By Bernstein-Zelevinsky \cite[Proposition 1.14]{BZ},
$V$ is equal to the space of all compactly supported, locally constant sections of this sheaf. We write this as 
\[ 
V=C_c(\Omega, \mathcal V). 
\] 
Let us look at this for $\Pi_2/\Pi_3$. Here $\bar{\mathcal N}_2/\bar{\mathcal N}_{3} = V_3 \otimes V_5$  and $\Omega$ is the set of non-zero 
pure tensors. Let $x\in \Omega$, and let $\omega_x$ be the $\SL_3$-orbit of $x$. If $\sigma$ is a subquotient of $V= \Pi_2/\Pi_3$, it is also a sub, and 
the composite of natural maps 
\[ 
\sigma \rightarrow C_c(\Omega, \mathcal V) \rightarrow  C_c(\omega_x, \mathcal V)
\] 
is non-zero for some $x\in \Omega$. The stabilizer of $x$ in $\SL_3$ is the derived group $P_x^{\rder}$ of a maximal parabolic subgroup $P_x=M_xU_x$. 
The stabilizer acts naturally on $\mathcal V_x$. Since the wavefront set of $\Pi$ is the minimal orbit, the unipotent radical $U_x$ must act trivially. Hence 
\[ 
C_c(\omega_x, \mathcal V)= \Ind_{P_x}^{\SL_3} (\ind_{M_x^{\rder}}^{M_x}( \mathcal V_x)). 
\] 
Since $\sigma$ is a submodule of this induced representation, by Frobenius reciprocity, it follows that $\sigma_{U_x} \neq 0$, a contradiction.  

\smallskip  
The subquotient  $\Pi_1/\Pi_2$ is treated in the same way. Indeed, $\bar{\mathcal N}_1/\bar{\mathcal N}_{2} = V_3^{\ast} \otimes \wedge^2 V_5^{\ast}$  
and $\Omega$ is the set of non-zero pure tensors $y\otimes z$ where $z$ is a highest weight vector in $ \wedge^2 V_5^{\ast}$.  

\smallskip 

It remains to deal with $\Pi_0/\Pi_1$. This quotient is the Jacquet functor $\pi_{\mathcal N}$. It is a finite-length 
($\SL_3 \times \SL_5$)-module generated by Iwahori-fixed vectors. Thus no supercuspidal $\SL_3$ representations can appear here. 

\smallskip 

Assume now that $M^{\rder} = \SL_4\times \SL_2$.  In this case we realize $P=MN$  so that $M$ is the centralizer of 
\[ 
h=\left(
\begin{array}{cccccc}  
-1 & 0 & 0 & 0 & 0 & 0 \\ 
0 & -1 & 0 & 0 & 0 & 0 \\ 
0 & 0 & -1 & 0 & 0 & 0 \\ 
0 & 0 & 0 & -1 & 0 & 0 \\ 
0 & 0 & 0 & 0 & 2 & 0 \\ 
0 & 0 & 0 & 0 & 0 & 2
\end{array} \right) \in \mathfrak{sl}_6, 
\] 
while the Lie algebra of $N$ is the eigenspace of $h$, for the adjoint action, with eigenvalue $3$. The desired $\mathcal P$ comes from $h$, viewed as an element of $\mathfrak{sl}_6\subset \mathfrak e$. Again, $\mathcal M$ is the centralizer of $h$. Note that 
$\mathcal M^{\rder}= \SL_3 \times \SL_4\times \SL_2$. The unipotent radical $\mathcal N$ is described using the 
$\mathbb Z$-grading $\mathfrak e= \oplus_i \mathfrak e_i$ coming from the adjoint action of $h$ on $\mathfrak e$. To work out the grading, observe that 
\[ 
\wedge^2 V_6= \wedge ^2 V_4 \otimes \mathbb C \oplus V_4 \otimes V_2 \oplus \mathbb C\otimes \wedge^2 V_2, 
\] 
as an ($\SL_4\times \SL_2$)-module, where $h$ acts on the three summands by eigenvalues $-2,  1$, and $4$, respectively.  
The unipotent radical of $\mathcal P$ is a $4$-step unipotent subgroup 
\[ 
\mathcal N=\mathcal N_1 \supset \mathcal N_2 \supset \mathcal N_3 \supset \mathcal N_4\supset \mathcal  N_5=0
\] 
such that $\mathcal N_i/\mathcal N_{i+1} \cong \mathfrak e_i$ for $1\leq i\leq 4$.  Each of these subquotients is naturally an $\mathcal M$-module.As ($\SL_3 \times \SL_4 \times \SL_2$)-modules, they are:  
\[ 
\mathfrak e_1= V_3 \otimes V_4\otimes V_2,  \, 
\mathfrak e_2= V_3^{\ast} \otimes \wedge^2 V_4^{\ast} \otimes \mathbb C ,  \, 
\mathfrak e_3 = \mathbb C \otimes V_4^{\ast}\otimes V_2, \, 
\mathfrak e_4 = V_3 \otimes  \mathbb C \otimes \wedge^2V_2. 
\] 
Observe that $\mathcal N_3/\mathcal N_{4} \cong \mathfrak e_3\cong N$. The proof here proceeds in the same way as for the other class of parabolics. 
We leave the details to the reader. 
\end{proof} 

\begin{prop} \label{P:support_2} 
 Let $\sigma$ be a supercuspidal representation of $\SL_3$. Let $\Theta(\sigma)$ be the lift to $E_6$. 
Let $P=MN$ be a maximal parabolic subgroup of $E_6$ such that $M^{\rder}$ is not $\SL_3 \times \SL_3 \times \SL_2$ or 
$\SL_6$. Then $\Theta(\sigma)_N=0$. 
\end{prop} 
\begin{proof} We repeat the proof of Proposition \ref{P:support_1}. Again, we need 
  a maximal parabolic subgroup $\mathcal P=\mathcal M\mathcal N$ in the ambient group that intersects $E_6$ in $P$. In this case,
\[ 
\mathfrak e = \mathfrak{sl}_3 \oplus \mathfrak{e}_6 \oplus V_3 \otimes V_{27}  \oplus V_3^{\ast} \otimes  V_{27}^{\ast}.
\] 
Assume first that $M^{\rder} = \Spin_{10}$.  Let $V_{16}$ and $V_{16}^{\ast}$ be two 16-dimensional  half-spin modules. 
As the notation indicates, they are dual to each other. We can pick the labeling so that 
$V_{16}\cong N$.  Let $V_{10}$ be the 
standard representation of $\Spin_{10}$. 
We pick a semisimple $h\in \mathfrak e_6$ such that $M$ is the centralizer of $h$, 
and the Lie algebra of $N$ is the eigenspace, for the adjoint action of $h$, with the eigenvalue $3$.  In this case $\mathcal M^{\rder} \cong \SL_3\times \Spin_{10}$. 
Under the restriction to $\Spin_{10}$, 
\[ 
V_{27} = V_{10}\oplus V_{16}  \oplus \mathbb C 
\] 
where $V_{10}$ is the standard representation of $\Spin_{10}$, 
and $h$ acts on the three summands by eigenvalues  $-2, 1$, and $4$, respectively. Thus $\mathcal N$ is a $4$-step unipotent subgroup with a filtration $\mathcal N_i$ such that 
 $\mathcal N_i/\mathcal N_{i+1} \cong \mathfrak e_i$ for   $1\leq i\leq 4$. 
  Each of these subquotients is naturally an $\mathcal M$-module. As ($\SL_3 \times \Spin_{10}$)-modules, they are:
\[ 
\mathfrak e_1= V_3 \otimes V_{16},  \, 
\mathfrak e_2= V_3^{\ast} \otimes  V_{10}^{\ast},  \, 
\mathfrak e_3 = \mathbb C \otimes V_{16}^{\ast}, \, 
\mathfrak e_4 = V_3 \otimes  \mathbb C. 
\] 
Observe that $\mathcal N_3/\mathcal N_{4} \cong \mathfrak e_3\cong N$. The proof here proceeds in the same way as in Proposition \ref{P:support_1}. 
We leave the details to the reader. 

\smallskip 
Finally, assume that $M^{\rder} = \SL_5\times \SL_2$. In this case $N$ is not abelian. It is a two-step unipotent group such that, as 
$\SL_5\times \SL_2$-modules, 
\[ 
N/N^{\rder} \cong \wedge^2 V_5^{\ast} \otimes V_2 \text { and }  N^{\rder} \cong V_5 \otimes \mathbb C. 
\] 
We let $h$ be a multiple of the fundamental co-weight so that it acts on the Lie algebra of $N^{\rder}$ with eigenvalue $6$.  In particular, the filtration of the unipotent radical $\mathcal N$ will be of length $6$, the most Lie algebras allow! 
Under the restriction to $\SL_5\times \SL_2$, 
\[ 
V_{27} = \wedge^2V_{5} \otimes \mathbb C \oplus V_5\otimes V_2 \oplus V_5^{\ast}\otimes \mathbb C   \oplus \mathbb C \otimes V_2
\] 
where $h$ acts on the four summands by eigenvalues $-2, 1, 4$ and $-5$,  respectively. 
Thus $\mathcal N$ is a $6$-step unipotent subgroup with a filtration $\mathcal N_i$ such that 
 $\mathcal N_i/\mathcal N_{i+1} \cong \mathfrak e_i$ for  $1\leq i \leq 6$. 
  Each of these subquotients is naturally an $\mathcal M$-module. As ($\SL_3 \times \SL_5\times \SL_2$)-modules:  
  \[ 
\mathfrak e_1= V_3 \otimes V_5\otimes V_2,  \, 
\mathfrak e_2= V_3^{\ast} \otimes \wedge^2 V_5 \otimes \mathbb C ,  \, 
\mathfrak e_3 = \mathbb C \otimes \wedge^2 V_5^{\ast}\otimes V_2 
\] 
\[ 
\mathfrak e_4 = V_3 \otimes  V_5^{\ast}  \otimes  \mathbb C, \, 
\mathfrak e_5= V_3^{\ast} \otimes \mathbb C  \otimes V_2  ,  \, 
\mathfrak e_6 = \mathbb C \otimes  V_5 \otimes \mathbb C 
\] 
Observe that $\mathcal N_3/\mathcal N_{4} \cong \mathfrak e_3\cong N/N^{\rder}$ and $\mathcal N_6 = N^{\rder}$. 
The proof here proceeds in the same way as before, so we leave the details to the reader. 
\end{proof}  

\section{$E_7$ Siegel parabolic} 
\label{S:Siegel}
 Let $\mathcal G$ be the split simply connected $p$-adic group of type $E_7$. Let $\mathcal P = \mathcal{MN}$ be the Siegel maximal parabolic subgroup of $\mathcal G$: the Levi $\mathcal M$ is of type $E_6$, and the unipotent radical $\mathcal N$ is a $27$-dimensional representation, isomorphic to the exceptional Jordan algebra $J$.
 
 The dual pair is $\SL_3 \times \SL_6$, which we embed into $\mathcal G$ so that $\SL_3$ is contained in $\mathcal M$. Since $\mathcal{G}$ is simply connected, $\Hom(\mathcal M, \GL_1) \cong \mathbb Z$ is generated by a fundamental weight $\varpi_{\mathcal M}$.  We identify the center of  $\mathcal{M}$ with $\GL_1$ so that a central element acts on $\mathcal N$ by $z^2$ and $\varpi_{\mathcal M}(z)=z^3$. 

The parabolic $\mathcal{P}$ intersects $\SL_6$ in $P=MN$, the maximal parabolic with derived Levi $\SL_3\times \SL_3$. Without loss of generality, we assume that $N$ is given by upper-triangular matrices. In this section, we compute the Jacquet module of the minimal representation of $\mathcal G$ with respect to $N$. Our starting point is the following (Theorem 1.1 \cite{MS97}):

\begin{prop}  \label{prop:msE6} Let $(\Pi, V)$ be the minimal representation of $\mathcal G$. Let $\Omega$ be the $\mathcal M$-orbit  of the highest weight vector in $\mathcal{N}$.  Let $\bar{\mathcal P}$ be the parabolic opposite to  $\mathcal P$, so that their intersection is $\mathcal M$. 
We have the following exact sequence of  $\bar{\mathcal P}$-modules, 
\[ 
0\rightarrow C_c^{\infty}(\Omega)\rightarrow V \rightarrow  V_{\bar{\mathcal N}}\rightarrow 0.
\]
The action of $\bar{\mathcal P}$ on  $f\in C_c^{\infty}(\Omega)$ is given by: 
\begin{itemize} 
\item For every  $\bar n\in \bar{\mathcal N}$
\[ 
\Pi(n)f(x)=\psi(\langle  x, \bar n\rangle )f(x).
\]  
\item For every  $m\in \mathcal M$ 
\[
\Pi(m)f(x)= |\varpi_{\mathcal M}(m)|^{-4} f(m^{-1} x m).
\]  
\end{itemize} 
Here $\psi$ is a non-trivial additive character, $\langle\cdot,\cdot\rangle$ is a pairing between $\mathcal N$ and $\bar{\mathcal N}$ given by the Killing form. 
Moreover, $V_{\bar{\mathcal N}} \cong V(\mathcal M)\otimes |\varpi_{\mathcal M}|^{-2} \oplus  |\varpi_{\mathcal M}|^{-4} $ where $V(\mathcal M)$ is the minimal representation of $\mathcal M$ with the center acting trivially. 
\end{prop} 

The above proposition implies that
\[ 
0\rightarrow C_c^{\infty}(\Omega_0)\rightarrow V_{\bar{N}}\rightarrow  V_{\bar{\mathcal N}}\rightarrow 0,
\]
where $\Omega_0$ is the intersection of $\Omega$ with the orthogonal complement of $\bar N$ in $\mathcal N$. We now describe these objects in more detail.

The Levi $M$ is 
$(\GL_3 \times \GL_3)^{\det}$, the group consisting of pairs $(g_2,g_3)$ such that $\det g_2 = \det g_3^{-1}$. We have 
\[ 
\varpi_M(g_2,g_3)= \det g_2 = \det g_3^{-1}.
\]
Note that $\varpi_{\mathcal M}$ restricted to $M$ gives $\varpi_M$.  
Under the action of $\SL_3 \times(\GL_3 \times \GL_3)^{\det}$, $J=\mathcal N$ decomposes as
\[
\Hom(V_2,V_1) \oplus \Hom(V_3,V_2) \oplus \Hom(V_1,V_3) 
\]
where $V_1,V_2,V_3$ are the standard representations for the three copies of $\SL_3$. Let $(A,B,C)$ denote the triple of homomorphisms living in the above space. The elements perpendicular to $\bar N$ in this setting are those with $B=0$. The orbit $\Omega_0$ consists of pairs of homomorphisms
\[
(A,C) \in\Hom(V_2,V_1)  \oplus \Hom(V_1,V_3)
\]
where rank of $A$ and $C$ is $1$, and $CA = 0$. An element $(h,g_2,g_3) \in
\SL_3 \times (\GL_3 \times \GL_3)^{\det}
$
acts on these pairs by
\[
(h, g_2, g_3)\cdot (A,C) = (\det(g_2)\cdot hAg_2^{-1}, \det(g_3)^{-1}\cdot g_3Ch^{-1}).
\]
Note the det twists of the geometric action; these ensure the center acts correctly. We have three orbits:
\begin{enumerate}[(i)]
\item $C=0$
\item $A=0$
\item $A,C\neq 0$.
\end{enumerate}
Fixing the standard basis for $V_1,V_2,V_3$, we view these operators as matrices. We write down the corresponding representations.

\textbf{Orbit (i)} is the orbit of $(A, 0)$ with $A = \begin{pmatrix}
0 & 0 & 1\\
0 & 0 & 0\\
0 & 0 & 0
\end{pmatrix} = \begin{pmatrix}
1\\
0\\
0
\end{pmatrix} \begin{pmatrix}
0 & 0 &  1
\end{pmatrix}$. The line through this element is stabilized by
the parabolic in $\SL_3 \times (\GL_3 \times \GL_3)^{\det}$ whose Levi consists of elements
\[
\left(
\begin{array}{c|cc}
	\det X^{-1} & 0 & 0 \\ \hline
	0 & \multicolumn{2}{c}{\multirow{2}{*}{$X$}} \\
	0 & \multicolumn{2}{c}{} 
\end{array}
\right), \quad \left(
\begin{array}{cc|c}
	\multicolumn{2}{c|}{\multirow{2}{*}{$Y$}} & 0 \\
	\multicolumn{2}{c|}{} & 0 \\ \hline
	0 & 0 & z
\end{array}
\right), \quad 
\begin{pmatrix}
\phantom{0} & \phantom{0} & \phantom{0}\\
\phantom{0} & W & \phantom{0}\\
\phantom{0} & \phantom{0} & \phantom{0}
\end{pmatrix} \quad \in \SL_3 \times \GL_3\times \GL_3
\]
with $z\cdot \det Y \cdot \det W=1$. This group acts on the line through $(A,0)$ by $\det X^{-1}\det Y$. The point is thus fixed by the subgroup consisting of triples as above with $\det Y=\det X$. In total, we get the representation
\[
\Ind_{P_2 \times P_1 \times \SL_3}^{\SL_3 \times M_1}(\delta_1\cdot C_c(F^\times))
\]
where $\delta_1$ is some character of the stabilizer.

\textbf{Orbit (ii)} is the orbit of $(0, C)$ with $C = \begin{pmatrix}
	0 & 0 & 1\\
	0 & 0 & 0\\
	0 & 0 & 0
\end{pmatrix} = \begin{pmatrix}
	1\\
	0\\
	0
\end{pmatrix} \begin{pmatrix}
	0 & 0 &  1
\end{pmatrix}$ as before. The line through this element is stabilized by
the parabolic in $\SL_3 \times (\GL_3 \times \GL_3)^{\det}$ whose Levi consists of elements
\[
\left(
\begin{array}{cc|c}
	\multicolumn{2}{c|}{\multirow{2}{*}{$X$}} & 0 \\
	\multicolumn{2}{c|}{} & 0 \\ \hline
	0 & 0 & \det X^{-1}
\end{array}
\right), \quad \begin{pmatrix}
	\phantom{0} & \phantom{0} & \phantom{0}\\
	\phantom{0} & Y & \phantom{0}\\
	\phantom{0} & \phantom{0} & \phantom{0}
\end{pmatrix}, \quad
\left(
\begin{array}{c|cc}
	z & 0 & 0 \\ \hline
	0 & \multicolumn{2}{c}{\multirow{2}{*}{$W$}} \\
	0 & \multicolumn{2}{c}{} 
\end{array}
\right) \quad \in \SL_3 \times \GL_3\times \GL_3
\]
with $z\cdot \det W \cdot \det Y=1$. This group acts on the line through $(0,C)$ by $\det X \cdot \det W^{-1}$. We get the representation 
\[
\Ind_{P_1 \times \SL_3 \times P_2}^{\SL_3 \times M_1}(\delta_2\cdot C_c(F^\times)),
\]
where $\delta_2$ is a character of the stabilizer.

\textbf{Orbit (iii)} is the orbit of $(A, C)$ with $A=C = \begin{pmatrix}
	0 & 0 & 1\\
	0 & 0 & 0\\
	0 & 0 & 0
\end{pmatrix} = \begin{pmatrix}
	1\\
	0\\
	0
\end{pmatrix} \begin{pmatrix}
	0 & 0 &  1
\end{pmatrix}$. The line through this element is stabilized by
\[
B \times P_1 \times P_2
\]
where $B$ is the Borel in $\SL_3$. The Levi of this parabolic consists of elements
\[
\begin{pmatrix}
t_1& \phantom{0} & \phantom{0}\\
	\phantom{0} & t_2 & \phantom{0}\\
	\phantom{0} & \phantom{0} & t_3
	\end{pmatrix},\quad 
\left(
\begin{array}{cc|c}
	\multicolumn{2}{c|}{\multirow{2}{*}{$X$}} & 0 \\
	\multicolumn{2}{c|}{} & 0 \\ \hline
	0 & 0 & y
\end{array}
\right),\quad 
	\left(
	\begin{array}{c|cc}
		z & 0 & 0 \\ \hline
		0 & \multicolumn{2}{c}{\multirow{2}{*}{$W$}} \\
		0 & \multicolumn{2}{c}{} 
	\end{array}
	\right) \quad \in \SL_3 \times \GL_3\times \GL_3
\]
with $t_1t_2t_3=1$ and $\det X\cdot  y = z^{-1}\cdot \det W^{-1}$. The point $(A,C)$ is fixed by the subgroup consisting of triples as above with $\det X=t_1^{-1}$ and $\det W=t_3^{-1}$. Thus we get a character $\delta_3$ of this stabilizer, and a regular representation $C_c(F^\times\times F^\times)$ where the above group acts on $F^\times\times F^\times$ by $t_1\cdot \det X$ on the first copy, and $\det W^{-1}\cdot t_3^{-1}$ on the second. In total, we get the representation
\[
\Ind_{B \times P_1 \times P_2}^{\SL_3 \times M_1}(\delta_3\cdot C_c(F^\times \times F^\times)).
\]
It remains to pin down the characters $\delta_1, \delta_2,\delta_3$. To determine $\delta_3$, note that we already know that the generic principal series $\chi_1 \times 1 \times \chi_3$ of $\SL_3$ lifts to $(\chi_1 \circ \det) \times 1 \times (\chi_3 \circ \det)$ of $\SL_6$ induced from the $(\GL_2 \times \GL_2\times \GL_2)^{\det}$ Levi. This forces $\delta_3$ to be the character
\[\left(
\left(
\begin{array}{cc|c}
	\multicolumn{2}{c|}{\multirow{2}{*}{$X$}} & 0 \\
	\multicolumn{2}{c|}{} & 0 \\ \hline
	0 & 0 & y
\end{array}
\right),\quad 
	\left(
	\begin{array}{c|cc}
		z & 0 & 0 \\ \hline
		0 & \multicolumn{2}{c}{\multirow{2}{*}{$W$}} \\
		0 & \multicolumn{2}{c}{} 
	\end{array}
	\right)\right) \mapsto |y|^{-1/2}\cdot |z|^{1/2}.
\]
The pieces (i) and (ii) are relevant to lifts of degenerate principal series of $\SL_3$ (induced from a character of $\GL_2$). Recall that $\Ind_{\GL_2}^{\SL_3} \chi$ lifts to the representation $(\chi \circ \det_{\GL_3}) \times \chi \times 1_{\GL_2}$ of $\SL_6$, by Proposition \ref{P:theta_nc}. The block-size $(3,1,2)$ Jacquet module of the latter representation contains precisely one irreducible constituent where $\SL_2$ and $\SL_3$ both act trivially: $(\chi \circ \det_{\GL_3}) \otimes  \chi \otimes 1_{\GL_2} $. But this is the same as $1_{\GL_3} \otimes 1 \otimes (\chi^{-1} \circ \det_{\GL_2}) $. So applying the relevant Jacquet functor to $$\Pi \mapsto  \Ind_{\GL_2}^{\SL_3} \chi \otimes (\chi \circ \det{}_{\GL_3}) \times \chi \times 1_{\GL_2},$$
we see that $(\chi \circ \det_{\GL_2}) \otimes [1_{\GL_3} \otimes 1 \otimes (\chi^{-1} \circ \det_{\GL_2}) ]$ should appear. This is precisely what we get from piece (i), so no twisting is needed: $\delta_1 =1$. We get $\delta_2 = 1$ analogously.

To summarize, the Jacquet functor $r_{\overline{P}}(\Pi)$ has a filtration of length $3$ where the top level is the Jacquet module with respect to $\bar{\cal N}$, the middle is a direct sum of pieces described in cases (i) and (ii), and the bottom level is described in (iii).


The following proposition handles the big theta lift from $\SL_6$ for the hermitian representations appearing in the degenerate series induced from the Levi $A_1\times A_3$.   
These are the trivial representation and representations induced from a 
character $\chi\circ\varpi_2$ where $\chi$ is unitary character. 
Let $\pi_0$ be this representation for $\chi=1$.  

\begin{prop} \label{P:theta_siegel}
For hermitian representations of $G=\SL_6$ in the degenerate $A_1\times A_3$ series, we have: 
\begin{itemize}
    \item The big theta lift of the trivial representation is the irreducible representation of $\SL_3$ with the infinitesimal character $(2,0,-2)$. 
    \item The big theta lift of the representation induced from a unitary character $\chi\neq 1$ is an irreducible representation of $\SL_3$ induced from $|\cdot| \times \chi \times |\cdot|^{-1}$. 
    \item For $\pi_0$ (taking $\chi=1$ in the above bullet), the big theta lift $\Theta(\pi_0)$ has the trivial representation as a quotient, and $\Theta(\pi_0)_{U,\psi}$ is at most one-dimensional. 
\end{itemize}
\end{prop}

\begin{proof}
 For the first two bullets, we already know that the listed representations are quotients of the appropriate big theta, hence it suffices to show that nothing else appears.

 First, we prove the Proposition when $\pi$ is trivial or $\pi_0$. 
 We realize $\pi$ as a submodule of $\Ind_{\bar P}^{G}(\tau)$ 
 for an appropriate representation $\tau$ of $M$: the two representations are submodules of the unramified degenerate principal series $\Ind_{\bar P}^G(s)$ for $s=-3, -1$, respectively. Thus $\Theta(\pi)$ is a quotient of 
    \[ 
    (\Pi\otimes \Ind_{\bar P}^{G}(\tau^{\vee}))_G \cong (r_{\bar P}(\Pi)\otimes \tau^{\vee})_M\cong (\Pi_{\bar N}\otimes (\rho_{\bar P}\cdot \tau)^{\vee})_M
    \]
where $\rho_{\bar P}= |\varpi_M|^{-3}$. Hence $\rho_{\bar P}\cdot \tau= |\varpi_M|^t$ where 
$t=-3-s= 0, 2$, respectively.  Now, looking at the action of the center of $M$, the contribution to 
\[ 
(\Pi_{\bar N}\otimes  |\varpi_M|^{-t})_M
\]
comes from the bottom, and the top piece of the filtration of $\Pi_{\bar N}$, respectively, where on $\SL_3$ we have the claimed representations. 
For the case $\pi_0$ we use the correspondence in $E_6$. The claim in the second bullet is entirely analogous to the first one, with the only contribution coming from the bottom of the filtration.
\end{proof}

\section{$E_8$ Heisenberg parabolic} \label{S:heisenberg}
Let $\mathcal G$ be a $p$-adic group of type $E_8$. 
Let $\mathcal P = \mathcal{MN}$ be the Heisenberg maximal parabolic subgroup in $\mathcal G$. Then $\mathcal{M}$ is of type $E_7$, 
and $\mathcal N$ is a Heisenberg group. Let $\mathcal Z$ be its center. 
Let $\varpi_{\mathcal M}$ be the fundamental weight corresponding to $\mathcal M$. This character of $\mathcal M$ naturally appears as the conjugation action of $\mathcal{M}$ on $\mathcal {Z}$. We identify the center of $\mathcal M$ with $\GL_1$ so that $\varpi_{\mathcal M}(z)=z^2$ for any 
$z\in \GL_1$. 

We embed the dual pair $\SL_3 \times E_6$ in $\mathcal G$ so that $\SL_3$ is 
contained in $\mathcal M$. 
It will be convenient to express the various characters that appear in this section using the character lattice of $E_6$. To that end, we let $\varpi_1, \dotsc, \varpi_6$ denote the fundamental weights of $E_6$ (using the Bourbaki numbering). Similarly, we let $\varpi_1'$ and $\varpi_2'$ denote the fundamental weights in $\SL_3$.

The parabolic $\mathcal{P}$ intersects $G=E_6$ in $P=MN$, the maximal parabolic with derived Levi $\SL_6$. It is also the Heisenberg parabolic of $G$. As in the previous section, our goal is to compute the Jacquet module of the minimal representation with respect to $N$. We start from the following (Theorem 6.1 \cite{MS97}):
\begin{prop}  \label{prop:msE8} Let $(\Pi, V)$ be the minimal representation of $E_8$. 
Let $\Omega$ be the $\mathcal M$-orbit  of the highest weight vector in $\mathcal{N}/\mathcal{Z}$.  Let $\bar{\mathcal P}$ be the opposite parabolic to $\mathcal P$, so that their intersection is $\mathcal M$. 
We have the following exact sequence of  $\bar{\mathcal P}$-modules, 
\[ 
0\rightarrow C_c^{\infty}(\Omega)\rightarrow V_{\bar{\mathcal{Z}}}\rightarrow  V_{\bar{\mathcal N}}\rightarrow 0.
\]
The action of $\bar{\mathcal P}$ on  $f\in C_c^{\infty}(\Omega)$ is given by: 
\begin{itemize} 
\item For every  $\bar n\in \bar{\mathcal N}$
\[ 
\Pi(n)f(x)=\psi(\langle  x, \bar n\rangle )f(x).
\]  
\item For every  $m\in \mathcal M$ 
\[
\Pi(m)f(x)= |\varpi_{\mathcal M}(m)|^{-5} f(m^{-1} x m).
\]  
\end{itemize} 
Here $\psi$ is a non-trivial additive character and $\langle\cdot,\cdot\rangle$ is a pairing between $\mathcal N/\mathcal Z$ and $\bar{\mathcal N}/\bar{\mathcal Z}$ induced by the Killing form. 
Moreover, 
\[ 
V_{\bar{\mathcal N}} \cong V(\mathcal M)\otimes |\varpi_{\mathcal M}|^{-3} \oplus  |\varpi_{\mathcal M}|^{-5}
\] 
where $V(\mathcal M)$ is the minimal representation of $\mathcal M$ with the center acting trivially. 
\end{prop} 

$\mathcal{P}$ intersects $G=E_6$ in $P=MN$, the maximal parabolic with derived Levi $\SL_6$. It is also the Heisenberg parabolic of $G$. The 
center $Z$ of $N$ is equal to the center $\mathcal Z$ of $\mathcal N$.  
In particular,  the above proposition implies that, for the Jacquet functor 
$V_{\bar N}$,  
\[ 
0\rightarrow C_c^{\infty}(\Omega_0)\rightarrow V_{\bar{N}}\rightarrow  V_{\bar{\mathcal N}}\rightarrow 0.
\]
where $\Omega_0$ is the intersection of $\Omega$ and the orthogonal complement of $\bar N/\bar Z$ in $\mathcal N/\mathcal Z$. 

We now go on to describe these objects in more detail. 
The Levi $M$ is isomorphic to $\GL_1\times_{\mu_2}\SL_6$. 
Let $V$ be the standard 6-dimensional representation of $\SL_6$ consisting of column vectors, and let $V^{\ast}$ be its dual, which we realize on the space of column vectors, but the action is given by precomposing by the transpose-inverse. 
Then $M$ acts on both $V$ and $V^{\ast}$ so that 
$z\in\GL_1$ acts as $z$ on both $V$ and $V^{\ast}$. The action of $M$ on 
each of these two 6-dimensional representations is not faithful, but it is on the direct sum of the two. Hence we can realize $M$ as a subgroup 
of $\GL_6\times\GL_6$, consisting of elements $h=(h_1,h_2)$ such that 
\[ 
h_1^{\top} h_2= \lambda \cdot 1_6 \text{ and } \det(h_1)=\lambda^3 =\det(h_2),  \]
where $1_6$ is the $6\times 6$ identity matrix. The fundamental character  $\varpi_M: M\rightarrow \GL_1$ (also $\varpi_2$ in the Bourbaki notation) is  
\[ \varpi_M(h)= \lambda. 
\] 
The character $\varpi_M$ is equal to the restriction of $\varpi_{\mathcal M}$ to $M$. The $\alpha_2$ co-root is the one-parameter group 
\[ 
\begin{pmatrix}
1& \phantom{0} & \phantom{0} & \phantom{0} & \phantom{0} & \phantom{0} \\
\phantom{0} & 1 & \phantom{0}& \phantom{0} & \phantom{0} & \phantom{0}\\
\phantom{0} & \phantom{0} & 1& \phantom{0} & \phantom{0} & \phantom{0} \\
\phantom{0}& \phantom{0} & \phantom{0} & t & \phantom{0} & \phantom{0} \\
\phantom{0} & \phantom{0} & \phantom{0}& \phantom{0} & t & \phantom{0}\\
\phantom{0} & \phantom{0} & \phantom{0} & \phantom{0} & \phantom{0} & t  
\end{pmatrix},\quad
 \begin{pmatrix}
t& \phantom{0} & \phantom{0} & \phantom{0} & \phantom{0} & \phantom{0} \\
\phantom{0} & t & \phantom{0}& \phantom{0} & \phantom{0} & \phantom{0}\\
\phantom{0} & \phantom{0} & t& \phantom{0} & \phantom{0} & \phantom{0} \\
\phantom{0}& \phantom{0} & \phantom{0} & 1 & \phantom{0} & \phantom{0} \\
\phantom{0} & \phantom{0} & \phantom{0}& \phantom{0} & 1 & \phantom{0}\\
\phantom{0} & \phantom{0} & \phantom{0} & \phantom{0} & \phantom{0} & 1  
\end{pmatrix}.   
\]

We have an action of $\SL_3 \times M$ on $\mathcal{N}/\mathcal{Z} = F+J+J+F$, 
which we now describe. 
We view the first copy of $J$ as
\[
\Hom(V_1,V_2) \oplus \Hom(V_2,V_3) \oplus \Hom(V_3,V_1)
\]
(see Section \ref{S:Siegel}) while the second (dual) copy consists of maps in the other direction:
\[
\Hom(V_1,V_3) \oplus \Hom(V_3,V_2) \oplus \Hom(V_2,V_1). 
\]
Here $V_1, V_2, V_3$ are three-dimensional vector spaces and $V_1$ is the standard representation of $\SL_3$, the centralizer of $G$. Since the centralizer of $\SL_3$ is $G$, we have 
\[ 
N/Z=(\mathcal N/\mathcal Z)^{\SL_3} = F + \Hom(V_2,V_3) + \Hom(V_3,V_2) +F.
\]
Now it is clear that the complement of $N/Z$ in $\mathcal N/\mathcal Z$
consists of pairs of homomorphisms
\[
(A,C) \in \Hom(V_1,V_2) \oplus \Hom(V_3,V_1)
\]
and
\[
(A^*,C^*)  \in \Hom(V_2,V_1) \oplus \Hom(V_1,V_3). 
\]
This quadruple is in $\Omega$ if all maps have rank $1$, and the maps compose to zero: $AC =0$, $C^*A^*=0$, etc. If we fix the standard basis for $V_1, V_2$, and $V_3$, then we can think of these maps as matrices. We form a $6\times 3$ matrix by putting $A$ and $C^*$ together, and a $3\times 6$ matrix by doing the same to $C$ and $A^*$. Thus we are looking at pairs $(X,Y) \in M_{6,3}(F) \times M_{3,6}(F)$ where
\[
XY = 0, \quad YX = 0, \quad \text{rank}(X) = \text{rank}(Y)=1.
\]
The action of $\SL_3\times M$ is now
\[
(g,h) \cdot (X,Y) = (h_1Xg^{-1}, gYh_2^{\top}), \quad\text{for } g \in \SL_3,\ h=(h_1,h_2)\in M.
\]
The situation is similar to the one from the previous section. The $\SL_3\times M$-orbits are
\begin{enumerate}[(i)]
	\item $Y=0$
	\item $X=0$
	\item $X,Y\neq 0$.
\end{enumerate}
Thus the action of $\SL_3\times M$ on $C_c(\Omega_0) \subset V_{\bar N}$ will have a submodule consisting of functions supported on the open orbit (iii) and the quotient will be a direct sum of two spaces, consisting of smooth, compactly supported functions on orbits (i) and (ii).

We write down the corresponding representations. It will be convenient to let $U$ and $u$ denote the highest weight vectors of the standard representation of $\SL_6$ and $\SL_3$, respectively; $V$ and $v$ will denote the highest weight vectors of the costandard representations.

\textbf{Orbit (i)} is the orbit of $(X, 0)$ with $X =Uv^t=\begin{pmatrix}
	1\\
	0\\
	0\\
	0\\
	0\\
	0
\end{pmatrix} \begin{pmatrix}
	0 & 0 &  1
\end{pmatrix}$. The line through this element is stabilized by
\[
P_2 \times Q_{1}.
\]
where $P_2$ is the standard (block upper-triangular) parabolic with block size $(2,1)$ (as before), and $Q_{1}$ the appropriate maximal parabolic of $M$. This group acts on the line through $X$ by $\det \cdot\varpi_1$. Here $\det$ denotes the $\GL_2$ determinant on the Levi of $P_2$, and $\varpi_1$ is the fundamental character of $M$. (Observe that the character $\det$ on the Levi of $P_2$, restricted to the torus, is given by the second fundamental weight $\varpi_2'$.)

In total, we get the representation
\[
\ind_{P_2 \times Q_1}^{\SL_3 \times M}(|\varpi_M|^{-5}\cdot C_c(F^\times))
\]
where the induction is not normalized. A typical (unramified) $\SL_3\times M$ quotient is 
\[ 
\Ind_{P_2}^{\SL_3}(|\det|^{s+3/2}) \otimes 
\Ind_{Q_1}^M ( |\varpi_M|^{-7/2} \cdot |\varpi_1|^s) 
\]
where $s\in \mathbb C$. We note that the center of $M$ acts on the second factor as $|z|^{s-7}$.  The trivial representation of $[M,M]$ appears as a quotient of the second factor if and only if $s=3$.  
In that case, the first factor is the irreducible degenerate principal series that lifts to the minimal representation of $G$.

\textbf{Orbit (ii)} is the orbit of $(0, Y)$ with $Y = uV^t =\begin{pmatrix}
	1\\
	0\\
	0
\end{pmatrix} \begin{pmatrix}
0 & 0 & 0 &	0 & 0 &  1
\end{pmatrix}$. The line through this element is stabilized by
\[
P_1 \times Q_{6}.
\]
where $P_1$ is the standard (block upper-triangular) maximal parabolic with block size $(1,2)$ (as before), and $Q_{6}$ the appropriate maximal parabolic in $M$. This group acts on the line through our point $Y$ by $\varpi_6 \cdot \det$. Observe that the character $\det$ on the Levi of $P_1$, restricted to the torus is given by the first fundamental weight $\varpi_1'$. 

In total, we get the representation
\[
\ind_{P_1 \times Q_6}^{\SL_3 \times M}(|\varpi_M|^{-5}\cdot C_c(F^\times))
\]
where the induction is not normalized.
A typical (unramified) $\SL_3\times M$ quotient is 
\[ 
\Ind_{P_1}^{\SL_3}(|\det|^{s+3/2}) \otimes 
\Ind_{Q_6}^M ( |\varpi_M|^{-7/2} \cdot |\varpi_6|^s) 
\]
where $s\in \mathbb C$. We note that the center of $M$ acts on the second factor as $|z|^{s-7}$.  The trivial representation of $[M,M]$ appears as a quotient of the second factor if and only if $s=3$.
In that case, the first factor is the (other) irreducible degenerate principal series that lifts to the minimal representation of $G$.

\textbf{Orbit (iii)} is the orbit of $(X, Y)$ with $X=Uv^{\top}$ and $Y=uV^t$ as above. The line through this element is stabilized by
\[
B \times Q_{1,6}.
\]
where $B$ is the standard (upper-triangular) Borel in $\SL_3$, and $Q_{1,6}$ is the intersection of the two parabolics from previous cases: $Q_{1,6}=Q_1\cap Q_6$. 
Letting $t_1,t_2,t_3$ be the standard coordinates for the diagonal torus in $\SL_3$, we see that $B\times Q_{1,6}$ acts on 
the line through $(X,Y)$ by $\varpi_2' \cdot \varpi_1 $ and $\varpi_1'\cdot\varpi_6$. In total, we get the representation
\[
\ind_{B \times Q_{1,6}}^{\SL_3 \times M}(|\varpi_M|^{-5}\cdot C_c(F^\times\times F^\times)),
\]
where induction is not normalized. A typical (unramified) $\SL_3\times M$ quotient is 
\[ 
\Ind_{B}^{\SL_3}(|\varpi'_1|^{s+3/2}\cdot |\varpi'_2|^{t+3/2} ) \otimes 
\Ind_{Q_{1,6}}^M ( |\varpi_M|^{-5/2} \cdot |\varpi_1|^t |\varpi_6|^s) 
\]
where $s,t \in \mathbb C$. We note that the center of $M$ acts on the second factor as $|z|^{s+t-5}$.  The trivial representation of $[M,M]$ appears as a quotient of the second factor for $s=t=5/2$. In this case, the first is an irreducible principal series with infinitesimal character 
\[(4,0,-4). \]

Summarizing, we have described the Jacquet functor $r_{\bar P}(\Pi)$ for the Heisenberg parabolic in $G$. It has a filtration of length 3 where the top piece is the Jacquet functor with respect to $\bar{\mathcal N}$, the middle piece is a direct sum of the spaces in the cases (i) and (ii) above, and the bottom piece is the space described in case (iii). 

The following proposition handles the big theta lift from $E_6$ for hermitian representations appearing in the degenerate series induced from the Levi $D_5$, that is, from a character $\chi\circ \varpi_1$. These are the trivial and the minimal representation (which appear as subquotients) and fully induced representations when $\chi$ is unitary. All such induced representations are irreducible, by \cite{W03}. 

\begin{prop} \label{P:heisenberg}
For hermitian representations of $G=E_6$ in the degenerate $D_5$ series, we have: 
\begin{itemize}
    \item The big theta lift of the trivial representation is the irreducible representation of $\SL_3$ with infinitesimal character $(4,0,-4)$. 
    \item The big theta lift of the minimal representation is a direct sum of two degenerate representations of $\SL_3$ with infinitesimal characters $(-3,1,2)$ and $(3,-1,-2)$. 
    \item The big theta lift of the representation induced from $\chi\circ \varpi_1$, where $\chi$ is a unitary character, is an irreducible representation of $\SL_3$ induced from $|\cdot|^2 \times \chi \times |\cdot|^{-2}$. 
\end{itemize}
\end{prop}
   
\begin{proof}
 We already know that the listed representations are quotients of the big theta, hence it suffices to show that nothing else appears. 
 We shall realize $\pi$ as a submodule of $\Ind_{\bar P}^{G}(\tau)$ 
 for an appropriate representation $\tau$ of $M$. Then
    $\Theta(\pi)$ is a quotient of 
    \[ 
    (\Pi\otimes \Ind_{\bar P}^{G}(\tau^{\vee}))_G \cong (r_{\bar P}(\Pi)\otimes \tau^{\vee})_M\cong (\Pi_{\bar N}\otimes (\rho_{\bar P}\cdot \tau)^{\vee})_M
    \]
where $\rho_{\bar P}= |\varpi_M|^{-11/2}$. First, we look at $\pi$ trivial,  minimal, or $\pi_0$, the induced representation in the third bullet with $\chi=1$.  These three representations are submodules of the unramified degenerate principal series $\Ind_{\bar P}^G(s)$ for $s=-11/2,-7/2,-5/2$, respectively \cite{HS20}.  Hence $\rho_{\bar P}\cdot \tau= |\varpi_M|^t$ where 
$t=-11/2-s= 0, -2,-3 $, respectively.  Now, looking at the action of the center of $M$, the contribution to 
\[ 
(\Pi_{\bar N}\otimes  |\varpi_M|^{-t})_M
\]
comes from the bottom, middle and the top piece of the filtration of $\Pi_{\bar N}$, respectively, where on $\SL_3$ we have the claimed representations. Note that in the last case, we have reduced to the case of correspondence in $E_7$, i.e.\ Proposition \ref{P:theta_siegel}.

We now consider the remaining representations in the third bullet. Again, to compactify notation, assume that $\chi$ is unramified, so that 
$\chi\circ \varpi_1= |\varpi_1|^{s}$ where $s\in \mathbb C$ is purely imaginary. Then $\pi$ is a submodule of $\Ind_{\bar P}(\tau_s)$ where 
$\rho_{\bar P}\tau_s$ is a quotient of 
\[ 
\Ind_{Q_{1,6}}(|\varpi_M|^{-5/2} |\varpi_1|^{s-7/2} |\varpi_6|^{5/2}). 
\]
The parameter $t=5/2$ causes reducibility of the above representation, and by induction in stages, $\tau_s$ itself is induced from the $Q_1$ parabolic. (The family $\tau_s$ reduces for $s=0$ and has a one-dimensional sub, so $\pi_0$ is a submodule of a representation induced from a character of $M$, which is what we just used). 
Recall that the center acts on the above induced representation by $|z|^{s-6}$. As long as $|z|^s$ is non-trivial, computing 
\[ 
(\Pi_{\bar N}\otimes  (\rho_{\bar P}\tau)^{\vee})_M
\]
avoids the top part of $\Pi_{\bar N}$, where the center acts by $|z|^{-6}$
and $|z|^{-10}$. The middle part has representations induced from $Q_1$, but the action of the center is wrong. (The middle part will also not contribute to $\Theta(\pi)_{U,\psi}$). Finally the bottom of the filtration contributes the claimed $\Theta(\pi)$. 
\end{proof}

\begin{rmk}The degenerate series $\Ind_{\bar P}^G(s)$ has one additional reducibility at $s=-1/2$. The irreducible submodule is the lift $\pi_{57}$ of the trivial representation of $\SL_3$; see Proposition \ref{P:triviallifts}.
   \end{rmk}
\section{Ping-pong}\label{S:tabletennis}
In this section we prove the main results of this paper, Theorems \ref{thm_main1} and \ref{thm_main2}. Suppose $\sigma$ and $\pi$ are representations of $\SL_3$ and $G$ such that $\sigma\otimes\pi$ is a quotient of the minimal representation $\Pi$. Both theorems are proved using the method of period ping-pong, which involves two sequences of various $\Hom$-spaces related to $\sigma$ and $\pi$ (see \S\ref{subs:pp}). Assuming that $\pi$ is hermitian, we will connect the two sequences, arriving at the desired conclusions. This is why we begin the section with a brief discussion of hermitian representations. For the second sequence to work, we need to assume that $\sigma$ does not appear in the boundary of the representation $\Pi_H$, which was described in \S\ref{subs:boundary}. The few boundary cases are taken care of in Lemma \ref{L:hermitian_boundary} and \ref{L:boundary}, before we turn our attention to the main case.

We begin with some general observations. Let $G$ be a split reductive group, and fix a Borel subgroup $B=TU$. Recall that an irreducible representation $\sigma$ of $G$ is hermitian if the complex conjugate $\bar \sigma$  is isomorphic to the dual $\sigma^{\vee}$. 
This is the same as $\sigma$ having a non-zero invariant hermitian form. Of course, unitary representations are hermitian.  Assume that $\sigma_U\neq 0$. 
Recall that 
\[ 
(\sigma_U)^{\vee} \cong (\sigma^{\vee})_{\bar U} 
\] 
where  $\bar U$  is the opposite unipotent radical. Let $w_0$ be an element in the normalizer of $T$ representing the longest Weyl group element. 
Then $\bar U=w_0 U w_0^{-1}$. Thus, if $\sigma$ is hermitian, then 
\[ 
(\sigma_U)^{\vee} \cong  w_0(\bar \sigma_{U})w_0^{-1}. 
\] 
Now specialize $B=TU$ to a Borel subgroup of $\SL_3$. Let $\chi$ be a character of $T$. We can write 
\[ 
\chi= \alpha |\cdot |^x \times  \beta |\cdot |^y \times \gamma |\cdot|^z 
\] 
where $x,y,z$ are three real numbers such that $x+y+z=0$, and $\alpha, \beta, \gamma$ are three unitary characters of $F^{\times}$.  They are unique up to a 
simultaneous twisting by a unitary character.  Let 
\[ 
|\chi|= (x,y,z). 
\] 
If the character $\chi$ appears in $\sigma_U$, and $\sigma$ is hermitian,  then $\chi'=w_0(\bar\chi^{-1}) w_0^{-1}$  also appears in $\sigma_U$. Now, 
\[ 
|\chi'|=(-z,-y,-x) 
\] 
is Weyl-conjugate to $|\chi|$ if and only if one of the three $x,y,z$ is 0.  This narrows down the possibilities for hermitian $\sigma$.  


By Proposition \ref{P:Theta_1}, $\Pi_{H_a}$ is isomorphic to a degenerate principal series for $G_2$, whose restriction to $\SL_3$ has a three-step filtration, 
where the bottom piece is a Gelfand--Graev representation while the top two pieces (the boundary) do not depend on $a$.  We have the following:  

\begin{lemma} \label{L:hermitian_boundary} Let $\sigma$ be an irreducible  hermitian representation of $\SL_3$ appearing as a quotient of  one of the two boundary pieces. Then 
\begin{itemize} 
\item If $G=\SL_6$ then $\sigma$ is induced irreducibly from $\chi=|\cdot |^2 \times  1 \times   |\cdot|^{-2}$, or 
 $\chi= |\cdot | \times  \lambda  \times  |\cdot|^{-1}$  where $\lambda\neq 1$ is a unitary character, or it is the trivial representation. 
 \item If $G=E_6$ then $\sigma$ is induced irreducibly from $\chi=|\cdot |^4 \times  1 \times   |\cdot|^{-4}$,  or 
 $\chi= |\cdot |^2 \times  \lambda  \times  |\cdot|^{-2}$  where $\lambda$ is a unitary character.  
 \end{itemize} 
\end{lemma} 
\begin{proof} The first terms in both bullets come from the top of the filtration. For $G=\SL_6$, the middle piece of the filtration is $\Ind_B^{\SL_3}(\rho^{1/2}_U\cdot C_c(F^{\times}))$. By Frobenius reciprocity, 
\[ 
\Hom_{\SL_3}( \Ind_B^{\SL_3}(\rho_U^{1/2} \cdot C_c(F^{\times})),  \sigma)\cong \Hom_{T}( \rho_U^{1/2} \cdot C_c(F^{\times}),  \sigma_{\bar U}). 
\] 
If $\chi$ is a quotient of $\rho_U^{1/2} \cdot C_c(F^{\times})$, then 
\[ 
\chi = |\cdot |^{1-t} \times \lambda |\cdot |^{2t} \times |\cdot |^{-1-t} \sim \bar \lambda |\cdot |^{1-t} \times |\cdot |^{2t} \times \bar\lambda |\cdot |^{-1-t} 
\] 
for some real number $t$. So $\sigma_{\bar U}$ must contain such a $\chi$. Now, since $|\chi|=(1-t,2t,-1-t)$, $\sigma$ can be hermitian only if $t=0$ or $t=\pm1$. If $t=0$ and 
$\lambda\neq 1$ we get the claimed family. If $\lambda=1$, then $(1,0,-1)$ is an exponent of the trivial representation, so $\sigma$ is trivial. If $t=\pm1$ then $|\chi|$ is conjugated 
to $(2,0,-2)$ and $\lambda$ must be trivial for $\sigma$ to be hermitian. This completes the first bullet; the second is analogous.  
\end{proof} 
As a consequence of the above lemma, a unitary $\sigma$, unless it is the trivial representation of $G=\SL_6$, is not a quotient of the boundary of $\Theta(1_H)$. 
\begin{lemma} \label{L:boundary}
Let  $\pi$  be a hermitian irreducible representation of $G$.  If $\pi$ is a lift of a representation $\sigma$ appearing in the boundary of $\Theta(1_H)$, then it is trivial, minimal, or a representation induced from a unitary character of a Levi of type $\SL_2\times \SL_4$ in $\SL_6$ or $D_5$ in $E_6$. 
\end{lemma} 
\begin{proof}  
We work this out in the $G=E_6$ case.  Let $Q$ and $\bar Q$ denote a $D_5$ maximal parabolic subgroup and its opposite. These two maximal parabolic subgroups are not conjugate; however, there is an intertwining operator between the corresponding degenerate principal series representations. 
We shall show that a $\sigma$ appearing in the boundary lifts to a quotient of one of the two degenerate series. Among these, the only hermitian representations 
 are precisely those from the statement of the lemma. This will complete the proof. 

We know that $\sigma_{\bar U}$ contains a character  
\[ 
\chi=  \lambda |\cdot |^{2-t} \times |\cdot |^{2t} \times \lambda |\cdot |^{-2-t} 
\] 
where $\lambda$ is a unitary character. If $\lambda \neq 1$, or if $\lambda=1$ and $t\neq 1, 1/3$ then $\chi$ is not on one of the reducibility walls and thus $\sigma_{\bar U}$ also contains the character 
\[ 
\chi=  |\cdot |^{2t}\times \lambda |\cdot |^{2-t}  \times \lambda |\cdot |^{-2-t} .
\] 
Thus, using induction in stages (observe that $\bar U\subset P$), $\sigma$ is a submodule of $\Ind_P^{\SL_3}(\tau)$ where $\tau$ is induced from the character 
(in $\GL_2$-coordinates) 
\[ 
 \lambda |\cdot |^{2-3t}  \times \lambda |\cdot |^{-2-3t} = \mu |\cdot |^2 \times \mu |\cdot |^{-2}, 
\] 
where $\mu = \lambda |\cdot |^{-3t}$.
Now, using the Frobenius reciprocity, if $\pi$ is a lift of $\sigma$, then $\tau\otimes \pi$ is a quotient of $r_P(\Pi)$.  Recall, by Theorem \ref{T:three_step}, that the top piece of $r_P(\Pi)$ 
is a finite-length $(\GL_2\times G)$-module, with $G$-factors the minimal and the trivial representation. If $\pi$ is a quotient of the middle piece, then it is 
a quotient of the degenerate principal series for $\bar Q$, as desired.  Finally, we look at the bottom piece. It includes the mini-theta from $\GL_2$ to $L$, the Levi of $Q$. 
Now, the lift of $\tau$ is the one-dimensional representation of $L$, the similitude character composed with $\mu$. 
Hence $\pi$ has to be a quotient of the degenerate series for $Q$. 
\end{proof}

\subsection{Ping-pong}\label{subs:pp}
Let $\sigma$ and $\pi$ be finite-length representations of $\SL_3$ and $G$ such that $\sigma\otimes\pi$ is a quotient of $\Pi$. Then, for every $a\in F^{\times}$, 
 we have the following sequences of inclusions
\[ 
\Hom_{U}(\sigma, \psi_a) \subseteq \Hom_{U}(\Theta(\pi), \psi_a) \simeq \Hom_G(\Pi_{U, \psi_a},\pi) \simeq \Hom_G(C_c(G/H_a), \pi) \simeq \Hom_{H_a}(\pi^{\vee}, \mathbb C), 
\] 
where the first is obvious, the second is a see-saw identity, and the last is Frobenius reciprocity. 
Assume that $\sigma$ does not appear in the boundary of $\Pi_{H_a}$. Then we have the second sequence 
\[ 
\Hom_{H_a}(\pi, \mathbb C)  \subseteq \Hom_{H_a}(\Theta(\sigma), \mathbb C) \simeq \Hom_{\SL_3} (\Pi_{H_a}, \sigma) \subseteq 
\Hom_{\SL_3} (\ind_U^{\SL_3}(\psi_{ab}), \sigma^{\vee}) \simeq \Hom_U(\sigma^{\vee}, \psi^{\vee}_{ab}) 
\] 
where the second is another see-saw identity, while the third uses the assumption on $\sigma$. 
Now, if $\pi$ is hermitian, then $\pi^{\vee} \cong \bar \pi$, thus
\[ 
\Hom_{H_a}(\pi^{\vee}, \mathbb C) \cong \Hom_{H_a}(\bar\pi, \mathbb C)\cong \Hom_{H_a}(\pi, \mathbb C) 
\] 
and the two sequences of inclusions connect. We keep this ping-pong rally going and run this sequence of inclusions starting with $\Hom_{U}(\sigma, \psi_{ab})$, ending with 
$\Hom_{U}(\sigma^{\vee}, \psi^{\vee}_{ab^2})$, and then starting with $\Hom_{U}(\sigma, \psi_{ab^2})$, ending with 
$\Hom_{U}(\sigma^{\vee}, \psi^{\vee}_{ab^3})=\Hom_{U}(\sigma^{\vee}, \psi^{\vee}_{a})$. By Lemma 2 of \cite{Pra19}, we have
\[ 
\dim \Hom_{U}(\sigma, \psi) = \dim \Hom_{U}(\sigma^{\vee}, \psi^{\vee}) 
\]  
for any Whittaker character $\pi$, so we see that all inclusions are isomorphisms!  In particular, we have 
\begin{equation} \label{E:ping_pong} 
\Hom_{U}(\sigma, \psi_a) \cong \Hom_{U}(\Theta(\pi), \psi_a) \cong \Hom_{H_a}(\pi,\mathbb C) \cong \Hom_{H_a}(\Theta(\sigma), \mathbb C) 
\end{equation} 
for all $a$, assuming that $\sigma$ does not appear in the boundary of $\Theta(1_{H})$ and $\pi$ is hermitian. However, for elegance's sake, we prove: 

\begin{lemma}\label{L:bcube} $b$ is a cube. 
\end{lemma} 
\begin{proof} If $b$ is not a cube then there exists a cubic character $\omega: F^{\times} \rightarrow \mathbb C^{\times}$ such that $\omega(b)\neq 1$.  The principal series of 
$\SL_3$ induced from the unitary character $1\times \omega\times \omega^2$ decomposes as a sum of three representations. The main result of \cite{Kuo02} 
states that, for each summand $\sigma$, 
\[ 
\dim \sigma_{U,\psi_a}=\dim \sigma_{U,\psi_{ac}}
\] 
for all $a\in F^{\times}$ and all $c$ in the kernel of $\omega$. As all three summands are generic for some choice of Whittaker characters, it follows that each summand has the Whittaker support on precisely one orbit, and the support defines a bijection between the summands and the orbits of the kernel of $\omega$. 
Hence, for a summand $\sigma$, there exists $a\in F^{\times}$ such 
that $\dim \sigma_{U,\psi_a}=1$ and $\dim \sigma_{U,\psi_{ab}}=0$. 
Clearly $\sigma$ is unitary, hence not a quotient of the boundary, and by Corollary \ref{C:unitary}, $\sigma$ lifts to a unitary representation $\pi$. Hence we can run the two 
sequences of inclusions, starting with $\Hom_U(\sigma,\psi_a)$,  which gives
\[ 
1=\dim \sigma_{U,\psi_a} \leq \ldots \leq \dim \sigma^{\vee}_{U,\psi^{\vee}_{ab}}=\dim \sigma_{U,\psi_{ab}}=0,
\] 
 a contradiction.  
\end{proof}

  Now we derive several consequences:  

\begin{thm}\label{thm_main1} Let $\sigma$ be an irreducible unitary Whittaker-generic representation of $\SL_3$.  Then $\Theta(\sigma)$ is irreducible. 
If $\sigma'$ is another irreducible unitary Whittaker-generic representation of $\SL_3$, perhaps for a different choice of Whittaker character, 
then $\Theta(\sigma)\cong \Theta(\sigma')$  implies 
$\sigma\cong\sigma'$.  
\end{thm}  
\begin{proof}  We know that $\Theta(\sigma)$ is unitarizable and that its non-cuspidal part is of finite length. Hence $\Theta(\sigma)$ is a direct sum of irreducible unitary representations.  Let 
$\pi$ be any non-zero direct summand of $\Theta(\sigma)$.  Let $\psi_a$ be a Whittaker character such that $\sigma_{U,\psi_a}\neq 0$ (and therefore $\dim \sigma_{U,\psi_a} =1$). 
Then, by (\ref{E:ping_pong}), 
\[  
\dim \Hom_{H_a}(\pi, \mathbb C)=\dim \Hom_{H_a}(\Theta(\sigma), \mathbb C)=1. 
\] 
Since $\pi$ is any non-zero direct summand, it follows that $\Theta(\sigma)$ must be irreducible. 
If $\Theta(\sigma)\cong\Theta(\sigma')\cong\pi$, then we also have  (\ref{E:ping_pong}) with $\sigma'$ in place of $\sigma$. It follows that $\sigma'$ is also 
generic for the same character $\psi_a$.  If $\sigma$ is not isomorphic to 
$\sigma'$, then $\sigma\oplus\sigma'$ is a quotient of $\Theta(\pi)$ and we have 
\[
2= \dim (\sigma_{U,\psi_a} )+ \dim (\sigma'_{U,\psi_a} )\leq \dim \Theta(\pi)_{U,\psi_a} =1,
\] 
a contradiction. 
\end{proof}

\begin{prop} Let $a\in F^{\times}$. 
Let  $\pi$  be a hermitian irreducible representation of $G$, but not one of those listed in Lemma \ref{L:boundary}. 
Then $\Hom_G(C_c(G/H_a), \pi)$ is at most one-dimensional. It is one-dimensional if and only if 
$\pi$ is a lift of a $(U,\psi_a)$-generic $\sigma$. 
\end{prop} 
\begin{proof} The usual see-saw 
\[ 
\Hom_{U}(\Theta(\pi), \psi_a)  \simeq \Hom_G(C_c(G/H_a), \pi) 
\] 
implies that $\Hom_G(C_c(G/H_a), \pi)$ is non-zero only if $\Theta(\pi)$ is non-zero. 
Since $\Theta(\pi)$ has finite length, it has an irreducible quotient $\sigma$. By the assumption on $\pi$, via Lemma \ref{L:boundary},  $\sigma$ is not a 
quotient of the boundary of $\Theta(1_{H_a})$, hence the isomorphisms (\ref{E:ping_pong}) hold. This implies the desired conclusion.  
\end{proof}

\begin{thm}\label{thm_main2}  Let $\pi$ be an irreducible hermitian representation of $G$. 
Then $\dim \Hom_H(\pi, \mathbb C) \leq 1$. It is 0 for the minimal representation (the case $G=E_6$). 
\end{thm} 
\begin{proof} For $\pi$ not in the boundary of $\Theta(1_H)$, this follows from the game of ping-pong above. For $\pi$ in the statement of Lemma \ref{L:boundary}, $\Theta(\pi)$ is computed in Propositions 
\ref{P:theta_siegel} and \ref{P:heisenberg}.  
\end{proof}

  \section{Lifts of tempered representations} \label{S:parameters} 
 
 Here we determine $\Theta(\sigma)$ for tempered $\sigma$ using a ``tower'' of theta correspondences, starting with the correspondence in 
 $E_6$ (Proposition \ref{P:gl3_correspondence}) as the base case. To that end, let $\mathcal G_i$ denote one of the
 simply connected  Chevalley groups in the sequence 
 \[ 
  E_7 < E_8 
 \] 
 where $i=7,8$ denotes the rank of the group.  
 Let $\mathscr P_i=\mathscr M_i \mathscr N_i $ be the (standard) parabolic in $\mathcal G_i$ such that 
 \[ 
 [\mathscr M_7, \mathscr M_7] =[\mathscr M_8, \mathscr M_8]
 \] 
 is the simply connected $E_6$. 
 The center $\mathscr Z_i$ of $\mathscr M_i$ is a connected one- and two-dimensional torus, respectively. The quotient 
 \[ 
 \mathscr M_7/\mathscr Z_7 =\mathscr M_8/\mathscr Z_8
 \] 
 is the adjoint $E_6$. To distinguish different cases, we write $G_i$ for the usual group $G$.
 
We can arrange the dual pair $\SL_3 \times G_i$ in $\mathcal G_i$ so that 
$\SL_3$ is contained in the Levi $\mathscr M_i$. Let $P_i=M_iN_i$ be the centralizer of $\SL_3$ in $\mathscr P_i$.  It is a parabolic subgroup in $G_i$ such that $[M_i,M_i]=\SL_3\times \SL_3$. Let $Z_i=\mathscr Z_i$ be the connected component of the center of $M_i$. Then, we claim, 
\[ 
M_i/Z_i \cong \SL_3 \times_{\mu_3} \SL_3. 
\]
In order to check the above claim, and for further use, we realize the root system $E_8$ inside $\mathbb R^8$ in the standard fashion, with simple roots  
 \[ 
 \alpha_1=\frac12(e_1+e_8)-\frac12(e_2+e_3+e_4+e_5+e_6+e_7), ~\alpha_2=e_1+e_2, 
 \] 
 \[ 
  \alpha_3=e_2-e_1, ~\alpha_4=e_3-e_2, ~\alpha_5=e_4-e_3, ~\alpha_6=e_5-e_4, ~\alpha_7=e_6-e_5, ~\alpha_8=e_7-e_6. 
 \] 
 The following is the Dynkin diagram of $E_6$, with the roots belonging to the $\SL_3\times \SL_3$ Levi marked black. 
\begin{picture}(200,120)(-120,-15)

\put(79,73){\line(0,-1){30}}
\put(79,40){\circle{6}}

\put(74,29){$\alpha_2$}

\put(02,82){$\alpha_1$}

\put(38,82){$\alpha_3$}

\put(74,82){$\alpha_4$}

\put(110,82){$\alpha_5$}

\put(146,82){$\alpha_6$}


\put(07,76){\circle*{6}}
\put(10,76){\line(1,0){30}}

\put(43,76){\circle*{6}}
\put(46,76){\line(1,0){30}}
\put(79,76){\circle{6}}

\put(82,76){\line(1,0){30}}
\put(115,76){\circle*{6}}

\put(118,76){\line(1,0){30}}
\put(151,76){\circle*{6}}

\end{picture}

Note that $\SL_6$ and its Levi $M_7$ are also visible in the above diagram. The center of the Levi $M_8$ contains one additional one-parameter group corresponding to the coroot $\alpha^{\vee}_2$. Thus 
\[
M_7/Z_7 =M_8/Z_8
\]
are the same, so it suffices to prove the claim for $i=7$. 
The Levi $M_7$ in $G_7=\SL_6$ is $(\GL_3\times\GL_3)^{\det}$. 
The connected component of its center is 
a one-dimensional torus given by block-diagonal matrices $(x,x^{-1})$ where $x$ denotes a $3\times 3$ scalar matrix. 
The quotient is indeed as claimed.  

Now observe that $\SL_3\times M_i/Z_i$ is a dual pair in 
$\mathcal M_i/\mathcal Z_i$ (tht is, adjoint $E_6$) where the correspondence, especially for generic representations, has been completely described; see Proposition \ref{P:gl3_correspondence}. The lift of $\sigma$ is a tempered representation that we denote by $\Theta_6(\sigma)$ (essentially $\sigma\otimes\sigma$). We note that a tempered representation $\sigma$ of 
$\SL_3$ is Whittaker-generic, and that $\Theta_6(\sigma)$ is tempered. 

\begin{prop}\label{P:tempered}
Let $\sigma$ be a tempered representation of $\SL_3$. Then the lift of 
$\sigma$ to $G$ is, respectively, the Langlands quotient 
\[ 
\Theta_7(\sigma)= \bar X(A_2\times A_2, \Theta_6(\sigma), 1) 
\text{ and } 
\Theta_8(\sigma)= \bar X(A_2\times A_2, \Theta_6(\sigma), (1,1)).
\]
\end{prop}
\begin{proof}
Let $\Pi_i$ be the minimal representation of $\mathcal G_i$. Then the Jacquet functor of $\Pi_i$ with respect to $\mathscr P_i$ contains the minimal representation $\Pi_6$ of the adjoint $E_6$ as a summand.  Hence 
$r_{P_i}(\Pi_i)$ has a quotient 
\[ 
\Pi_6\otimes \chi_i 
\]
for some $\mathbb R^+$-valued character $\chi_i$ of $M_i$. By Frobenius reciprocity, and using the fact that $\Theta_i(\sigma)$ is irreducible, we have an inclusion 
\[ 
\Theta_i(\sigma)\subset \Ind_{P_i}^{G_i}(\Theta_6(\sigma) \otimes\chi_i). 
\]
 Computing $\chi_i$ is an exercise keeping track of several $\rho$-shifts. However, the answer is of course independent of $\sigma$, and we know it 
 for the Steinberg representation, whence the general answer follows. 
\end{proof} 

  \appendix
\section{$\SL_3$-orbits on $G_2$ flag varieties} \label{S:orbits} 

We use a realization of the split octonion algebra $\mathbb O$ over $F$ (\cite{KMRT}, p.~507) as a vector space of $2\times 2$ matrices 
\[ 
x=
\left(\begin{matrix} 
\alpha & a \\
b & \beta \\
\end{matrix}\right) 
\] 
where $\alpha ,\beta \in F$ and $a,b\in F^3$. We think of $F^3$ as the set of column matrices, so that $\SL_3(F)$ acts on $F^3$ by right multiplication. On $F^3$ we have the standard 
dot product given by matrix multiplication  $a\cdot b= a^{\top} b $ where $a^{\top}$ is the transpose of $a$.  The octonion multiplication is given by 
\[ 
\left(\begin{matrix} 
\alpha & a \\
b & \beta  \\
\end{matrix}\right) 
\left(\begin{matrix} 
\gamma & c \\
d & \delta  \\
\end{matrix}\right) = 
\left(\begin{matrix} 
\alpha\gamma + a\cdot d & \alpha c + \delta a - b\wedge d\\
\gamma b +\beta d + a\wedge c & \beta \delta+ b\cdot c \\
\end{matrix}\right). 
\] 
where $\wedge$ denotes the cross product. Conjugation is given by 
\[ 
\left(\begin{matrix} 
\alpha & a \\
b & \beta \\
\end{matrix}\right) 
\mapsto 
\left(\begin{matrix} 
\beta & -a \\
-b & \alpha \\
\end{matrix}\right) 
\] 
and the norm by 
\[ 
N \left(\begin{matrix} 
\alpha & a \\
b & \beta \\
\end{matrix}\right) = \alpha\beta -a\cdot b. 
\] 
Let $G$ be the group of automorphisms of $\mathbb O$. It is a split exceptional group of type $G_2$. Its flag variety can be realized as a set of certain flags in the 
7-dimensional representation on the space of traceless octonions $\mathbb O^0$.  More precisely, a nil-subalgebra $V$ of $\mathbb O$ is a subspace of trace-$0$ octonions such that $xy=0$ for any two $x,y\in V$.  The dimension of such a non-zero $V$ is one or two.  Indeed, a one-dimensional nil-subalgebra $V_1$ is spanned by $x$ such that $x^2=0$ 
(observe  that $x^2= -N(x)$ for traceless $x$) 
one checks that for such $x$ 
\[
V_3 := \Delta\langle x\rangle = \{ y\in \mathbb O ~|~xy=0\} 
\] 
is three-dimensional. Furthermore, $V_3$ is not a nil-subalgebra but any two-dimensional subspace $V_2$ containing $V_1$ is. 
Now, the two partial flag varieties $\mathcal P_1$ and $\mathcal P_2$ for $G$ are the sets of one- and two-dimensional nil-subalgebras. 
The flag variety is the set of pairs $V_1\subseteq V_2$ of one- and two-dimensional nil-subalgebras.  The corresponding Borel subalgebra stabilizes a full flag in $\mathbb O^0$ given by 
\[ 
V_1 \subset V_2 \subset V_3 \subset V_3^{\perp} \subset V_2^{\perp} \subset V_1^{\perp} 
\] 
where $V_3 = \Delta V_1$ and $V^{\perp}$ is the orthogonal complement of $V$ with respect to the norm form. 

We have an embedding of $\SL_3(F)$ into $G$ such that $g\in \SL_3(F)$ acts on off-diagonal terms $a$ and $b$ by the standard action on $a$ and by transpose-inverse on $b$. Now we want to describe $\SL_3(F)$-orbits on flag varieties for $G$. 
Consider the partial flag $\mathcal P_1$. Each one-dimensional nil-subalgebra is spanned by a non-zero isotropic element 
\[ 
\left(\begin{matrix} 
\gamma & a \\
b & -\gamma \\
\end{matrix}\right), 
\] 
that is, $-\gamma^2-a\cdot b=0$.  The following is an easy check using the explicit action of $\SL_3(F)$.  

\begin{prop}  The group $\SL_3(F)$ acts on $\mathcal P_1$ with the following orbits: 
\begin{itemize} 
\item 5-dimensional open orbit $\gamma\neq 0$;
\item 4-dimensional orbit $ a\cdot b=0$,  $a\neq 0$ and $b\neq 0$;
\item two closed 2-dimensional orbits $a=0$ or $b=0$.
\end{itemize}  
\end{prop} 

\begin{prop}  Let $\mathcal B$ be the full flag variety for $G$. Let $\pi_1: \mathcal B \rightarrow \mathcal P_1$ be the natural projection. 
Let $O \subset \mathcal P_1$ be an $\SL_3(F)$-orbit. Then $\pi_1^{-1}(O)$ is a unique $\SL_3(F)$-orbit, except when $O$ is 4-dimensional. 
Then $\pi_1^{-1}(O)$ is a union of two orbits, one 5-dimensional, the other 4-dimensional. 
\end{prop} 
\begin{proof} 
Observe that the fiber of $\pi_1$ at $V_1 \in O$ is the projective line $\mathbb P(V_3/V_1)$ where $V_3=\Delta V_1$.  The stabilizer in 
$\SL_3(F)$ of the line $V_1$ acts naturally on this projective line, and the number of orbits in $\pi_1^{-1}(O)$ is equal to the number of orbits of the stabilizer on the projective line. If the stabilizer contains an $SL_2(F)$ then $\pi_1^{-1}(O)$ is one $\SL_3(F)$ orbits. This happens for all orbits except the $4$-dimensional, in which case we need the following
computation. 

Let $i,j,k$ be the usual orthonormal basis of $F^3$ so that the cross product takes the form $i\wedge j=k$, $k\wedge i=j$ and $j\wedge k =i$. 
Assume that $V_1$ is spanned by 
\[ 
\left(\begin{matrix} 
0 & i \\
k & 0\\
\end{matrix}\right).
\] 
Next, we need to compute $V_3=\Delta V_1$. If 
\[ 
\left(\begin{matrix} 
0 & i \\
k & 0\\
\end{matrix}\right) 
\left(\begin{matrix} 
\gamma & c \\
d & -\gamma  \\
\end{matrix}\right) = 
\left(\begin{matrix} 
 i\cdot d & -\gamma i - k\wedge d\\
\gamma k + i\wedge c &  k\cdot c \\
\end{matrix}\right) =0,
\]
then $c= (\alpha, -\gamma, 0)$  and $d= (0, \gamma, \beta) $.  Thus $V_3$ is an explicit vector space, and it is easy to see that the stabilizer of $V_1$ in 
$\SL_3(F)$ acts with two orbits on $\mathbb P(V_3/V_1)$, fixing $V_2/V_1$, where $V_2\subset V_3$ is given by $\gamma=0$. 
\end{proof} 

\begin{prop}\label{P:orbtis_3}  The group $\SL_3(F)$ acts on $\mathcal P_2$ with two orbits. The unique closed orbit $O_c$ consists of  
\[ 
V_2= \left\{ \left(\begin{matrix} 
0 & \alpha a  \\
\beta b & 0\\
\end{matrix}\right) ~|~ \alpha, \beta \in F \right\}
\] 
where $a,b\in F^3$ are non-zero such that $a\cdot b=0$.  This orbit is a flag variety for $\SL_3(F)$. 
\end{prop} 
\begin{proof} It is clear that $O_c$ is a closed $\SL_3(F)$-orbit isomorphic to the flag variety for $\SL_3(F)$. 
Let $\pi_2: \mathcal B \rightarrow \mathcal P_2$  be the natural projection. We need to understand how five $\SL_3(F)$-orbits in $\mathcal B$ are mapped. 
Let $V_2$ in $O_c$. We can complete $V_2$ to a full flag by taking a line $V_1\subset V_2$. Three choices 
\[ 
(\alpha,\beta) = (1,0), (0,1 ) \text { and }  (1,1) 
\] 
give flags belonging to different $\SL_3$-orbits, since the three $V_1$ are not $\SL_3(F)$-conjugate.  Thus $\pi_2^{-1}(O_c)$ consists of three orbits.  Now we pick a $V_2$ 
outside $O_s$. To that end, let 
\[ 
x= 
\left(\begin{matrix} 
1 & -j\\
j &  -1\\
\end{matrix}\right). 
\] 
Then $\Delta\langle x \rangle$ consists of matrices 
\[ 
\left(\begin{matrix} 
\gamma & c \\
d & -\gamma \\
\end{matrix}\right), 
\] 
where $d=(\alpha,\gamma,\beta )$ and $c=(\beta,-\gamma, -\alpha)$.  Let 
\[ 
x'=
\left(\begin{matrix} 
0 & i \\
k &  0\\
\end{matrix}\right) \in \Delta\langle x \rangle
\] 
and  $V_2=\langle x,x'\rangle$. Clearly $V_2$ is not in $O_c$. Furthermore, the lines $\langle x \rangle$ and $\langle x' \rangle$ are not $\SL_3(F)$-conjugate. Thus the largest two 
$\SL_3(F)$-orbits in $\mathcal B$ map to the same orbit in $\mathcal P_2$. 
\end{proof}  

We shall need stabilizers in $\SL_3(F)$ of points in the two orbits in $\mathcal P_2$.   In the closed orbit we take 
\[ 
V_2= \left\{ \left(\begin{matrix} 
0 & \alpha i  \\
\beta k & 0\\
\end{matrix}\right) ~|~ \alpha, \beta \in F \right\}. 
\] 
The stabilizer is the Borel subgroup $B=TU$ of upper-triangular matrices in $\SL_3$.  In the open orbit we take 
\[ 
V'_2= \left\{ \left(\begin{matrix} 
1 & \alpha i -\beta j \\
\beta j + \alpha k& 1\\
\end{matrix}\right) ~|~ \alpha, \beta \in F \right\}. 
\] 
The stabilizer here is a subgroup $B_0=T_0U_0$ of $B$ consisting of matrices 
\[ 
 \left(\begin{matrix} 
a & x & z  \\
0 & 1 & -x \\
0 & 0  & a^{-1} \\ 
\end{matrix}\right).
\] 

\section{Mini theta} 
\label{S:mini}
We describe the theta correspondence between similitude groups $\GL_2$ and $\almostGSpin_{2n}$ for $n$ odd.

\subsection{The group $\almostGSpin_{2n}$}
The group $\almostGSpin_{2n}$ we are interested in is obtained by gluing a copy of $\mathbb{G}_m$ to $\Spin_{2n}$ along the center $\mu_4$. To describe it, we need the root datum of $\Spin_{2n}$, for $n$ odd.

Let $T_0$ denote a maximal torus in $\Spin_{2n}$. The set of roots can be realized as $\Phi = \{\pm e_i \pm e_j: i\neq j\}$. Let $
R = \text{span}\ \Phi$
be the root lattice. The usual choice of simple roots is $\Delta = \{\alpha_1, \dotsc, \alpha_n\}$, where
\[
\alpha_i = e_i-e_{i+1} \text{ for }1\leq i\leq n-1,\quad \text{and }\quad \alpha_{n}=e_{n-1}+e_n.
\]
We let $P$ denote the weight lattice. The fundamental weights are
\[
\varpi_i = e_1 + \dotsb + e_i \text{ for }1\leq i\leq n-2,\quad \varpi_{n-1} = (\frac 12, \dotsc, \frac 12, -\frac 12), \quad \varpi_{n} = (\frac 12, \dotsc, \frac 12, \frac 12).
\]
Since $\Spin_{2n}$ is semisimple, the character lattice is equal to the weight lattice: $X^*(T_0)=P$ (dually, the cocharacters are $X_*(T_0)=R^\vee$). The characters of the center can thus be identified with $P/R$. One checks immediately that 
\[
\varpi_{n-1}\equiv -\varpi_n, \quad \varpi_1 \equiv \varpi_3 \equiv \dots \equiv \varpi_{n-2}=2\varpi_{n}, \text{ and } \quad \varpi_2 \equiv \varpi_4 \equiv \dots \equiv \varpi_{n-3}=0 \pmod R.
\]
Thus $P/R$ is (non-canonically) isomorphic to $\mathbb{Z}/4\mathbb{Z}$. We fix the isomorphism which sends $\varpi_n \pmod R$ to $1$ and denote by $[p]$ the image of $p\in P$ in $\mathbb{Z}/4\mathbb{Z}$.

We now consider the reductive group $\almostGSpin_{2n}$, which we define as the quotient
\[
\mathbb{G}_m \times_{\mu_4} \Spin_{2n},
\]
where $\mu_4$ is embedded anti-diagonally: $z \mapsto (z,z^{-1})$. We let $T=\mathbb{G}_m \times_{\mu_4} T_0,$ denote the corresponding torus. We identify the character lattice of the torus $\mathbb{G}_m \times T_0$ with $\mathbb{Z} \times P$. Those characters that are trivial on $\mu_4$ descend to $T$. A character $(a,p) \in \mathbb{Z} \times P$ takes $(z,z^{-1})$ to $z^{a-[p]}$. We thus get the following description of the character lattice:
\[
X^*(T) = \{(a,p) \in \mathbb{Z} \times P: a \equiv [p] \pmod 4\}.
\]
Dually, the cocharacter lattice $X_*(T)$ contains $\mathbb{Z} \times R^\vee$ as an index four subgroup. Modulo $\mathbb{Z} \times R^\vee$, it is generated by $((-1)^{\frac{n+1}{2}}/4, \varpi_{n}^\vee)$.

\subsection{Levi subgroups}
We want to describe two Levi subgroups of $\almostGSpin_{2n}$: the one corresponding to $\alpha_1$, and the one corresponding to the branch node $\alpha_{n-2}$.

\subsubsection{The branch Levi} Let $P_b=M_bN_b$ denote the parabolic corresponding to the branch node, $\alpha_{n-2}$.

To describe $M_b$, we use the (half-spin) representation $V^+$ of $\almostGSpin_{2n}$ whose highest weight is $(1,\varpi_n)$. The weights of $V^+$ are $(1,p)$, where 
\[
p=(\pm \frac{1}{2},\dotsc,  \pm \frac{1}{2})
\]
with an \textbf{even} number of negative signs. The derived subgroup of $M_b$ is $\SL_{n-2} \times \SL_2 \times \SL_2$. The $\SL_{n-2}$ corresponds to roots $\alpha_1, \dotsc, \alpha_{n-3}$; to distinguish the two copies of $\SL_2$, we denote them $\SL_2^-$ (corresponding to $\alpha_{n-1}$) and $\SL_2^+$ (corresponding to $\alpha_{n}$). 

\begin{lemma}$M_b/\SL_3 \cong \GL_2 \times \GL_2$.
\end{lemma}

\begin{proof}
We would like to compute the $\SL_{n-2}$ invariants of $V$. This amounts to finding the weights of $V^+$ orthogonal to roots $\alpha_1, \dotsc, \alpha_{n-3}$. We get weights of the form $(1,p)$, with $p$ one of the four possibilities (writing $\pm$ in place of $\pm 1/2$):
\[
\lambda^-=(-, \dotsc, -, +,-), \mu^-=(-, \dotsc, -, -,+), \lambda^+=(+, \dotsc,+, +,+), \mu^+=(+, \dotsc,+, -,-).
\]
This already shows that $(V^+)^{\SL_3}=V_2^- \oplus V_2^+$, where $V_2^\epsilon$ is the standard representation of $\SL_2^\epsilon$, with weights $(1,\lambda^\epsilon), (1,\mu^\epsilon)$ for $\epsilon \in\{\pm\}$.

An element of $M_b$ can be uniquely decomposed as a product of elements of $[M_b,M_b]$ and one-parameter groups corresponding to cocharacters $\alpha_{n-2}^\vee$ and $((-1)^{\frac{n+1}{2}}/4, \varpi_n^\vee)$. We have
\begin{alignat*}{2}
\langle(0, \alpha_{n-2}^\vee),(1,\lambda^-)\rangle &= -1, &\quad  \langle(0, \alpha_{n-2}^\vee),(1,\mu^-)\rangle &= 0,\\
\langle(0, \alpha_{n-2}^\vee),(1,\lambda^+)\rangle &= 0, &\quad \langle(0, \alpha_{n-2}^\vee),(1,\mu^+)\rangle &= 1.
\end{alignat*}
and
\begin{alignat*}{2}
\langle((-1)^{\frac{n+1}{2}}/4, \varpi_n^\vee),(1,\lambda^-)\rangle &= \frac{2+(-1)^{\frac{n+1}{2}}-n}{4}, &\quad  \langle((-1)^{\frac{n+1}{2}}/4, \varpi_n^\vee),(1,\mu^-)\rangle &=   \frac{2+(-1)^{\frac{n+1}{2}}-n}{4},\\
\langle((-1)^{\frac{n+1}{2}}/4, \varpi_n^\vee),(1,\lambda^+)\rangle &=  \frac{n+(-1)^{\frac{n+1}{2}}}{4}, &\quad \langle((-1)^{\frac{n+1}{2}}/4, \varpi_n^\vee),(1,\mu^+)\rangle &= \frac{n-4+(-1)^{\frac{n+1}{2}}}{4}.
\end{alignat*}
Thus the one-parameter subgroup generated by $\alpha_{n-2}^\vee(t)$ acts on $V_2^- \oplus V_2^+$ as
\[
\begin{pmatrix}
t^{-1} & 0\\
0 & 1
\end{pmatrix} \times \begin{pmatrix}
1 & 0\\
0 & t
\end{pmatrix}.
\]
Letting $z=\dfrac{2+(-1)^{\frac{n+1}{2}}-n}{4}\in \mathbb{Z}$ to simplify notation, $((-1)^{\frac{n+1}{2}}/4, \varpi_n^\vee)(u)$ acts by
\[
\begin{pmatrix}
u^z & 0\\
0 & u^z
\end{pmatrix} \times \begin{pmatrix}
u^{z+\frac{n-1}{2}} & 0\\
0 & u^{z+\frac{n-3}{2}}
\end{pmatrix}.
\]
The product of these two elements is
\[
\begin{pmatrix}
t^{-1}u^z & 0\\
0 & u^z
\end{pmatrix} \times \begin{pmatrix}
u^{z+\frac{n-1}{2}} & 0\\
0 & tu^{z+\frac{n-3}{2}}
\end{pmatrix}.
\]
The determinants of these two matrices are $t^{-1}u^{2z}$ and $tu^{2z+n-2}$.

It remains to see that the system $t^{-1}u^{2z}=a$ and $tu^{2z+n-2}=b$ has a unique solution for any choice of $a,b$. Multiplying the two equations, we get $u^{4z+n-2}=ab$. But $4z+n-2=(-1)^\frac{n+1}{2}$, so this equation is really $u^{\pm 1}=ab$, which certainly has a unique solution, after which $t=u^{2z}/a$ is also unique.
\end{proof}

\subsubsection{The $\alpha_1$ Levi}

Let $P_1=M_1N_1$ denote the parabolic corresponding to $\alpha_1$. The characters of $M_1$ are those elements of $X^*(T)$ vanishing on the coroots $\alpha_2,\dotsc, \alpha_n$.
If $p=\sum_{i=1}^n c_i\varpi_i$ is an element of $P$, then
$\langle p, \alpha_i^\vee\rangle= c_i$. Thus $(a,p) \in X^*(T)$ vanishes on $\alpha_2,\dotsc, \alpha_n$ if $c_2,\dotsc,c_n=0$. Recall moreover that $[\varpi_1] = 2[\varpi_n]$. This means we can identify the characters of $M_1$ as
\[
X^*(M_1)= \{(a,c\varpi_1):  a\equiv 2c \pmod 4 \}.
\]
In other words, we have
\[
X^*(M_1) = \mathbb{Z}(2,\varpi_1) \oplus \mathbb{Z}(2,-\varpi_1).
\]
These two generators of $X^*(M_1)$ are precisely the two characters one obtains from decomposing the ($2n$-dimensional) standard representation of $\almostGSpin_{2n}$ (with highest weight $(2,\varpi_1)$) as $2n=1+(2n-2)+1$.
\subsection{The correspondence}
With the above description of Levi subgroups, we are in position to describe the lifts of representations from $\GL_2$.

\begin{prop} Let $\pi$ be an irreducible representation of $\GL_2$.
\begin{enumerate}
\item If $\pi = \rho$ is supercuspidal, then $\Theta(\rho)=\Ind_{P_b}^G(\rho\nu^{1/2} \boxtimes \rho\nu^{-1/2})$. Here $\rho \boxtimes \rho$ is a representation of $\GL_2 \times \GL_2$ inflated to a representation of $M_b$.
\item If $\pi = \chi_1 \times \chi_2$ is a principal series, then $\Theta(\chi_1 \times \chi_2)=\Ind_{P_1}^G(\chi_1 \otimes \chi_2)$. Here $\chi_1 \otimes \chi_2$ is the character of $M_1$ given by $\chi_1 \circ (2,\varpi_1) \otimes \chi_1 \circ (2,-\varpi_1)$. The representation $\Ind_{P_1}^G(\chi_1 \otimes \chi_2)$ reduces in two cases:
\begin{enumerate}
\item $\chi_1/\chi_2=|\cdot|^{\pm 1}$. The two irreducible constituents of $\Ind_{P_1}^G(\chi_1 \otimes \chi_2)$ are the lifts of the respective constituents of $\chi_1 \times \chi_2$: the (twisted) trivial and Steinberg representations of $\GL_2$.
\item $\chi_1/\chi_2 =|\cdot|^{\pm (n-1)}$. Here $\chi_1 \times \chi_2$ is irreducible. The full theta lift $\Ind_{P_1}^G(\chi_1 \otimes \chi_2)$ is of length two and contains the (twisted) trivial representation.
\end{enumerate}
\end{enumerate}
\begin{proof}
This follows from \cite[Lemma 4.3]{HS25}, which details the correspondence between $\SL_2$ and $\Spin_{2n}$; the key idea is to use the Fourier-Jacobi functor as in [Lemma 4.1, op.cit.], which (in our case) results in the regular representation on $\GL_2$.
\end{proof}
\begin{rmk}
For convenience, we explicate the $n=3$ case of the above proposition, since it is this mini-theta that appears in the Jacquet module for the minimal representation of $E_7$; see Theorem \ref{T:three_step}. In this case $\almostGSpin_6=\GL_4$, and the mini-theta correspondence can be restated neatly: 
\begin{itemize}
    \item If $\pi=\chi\nu^{1}\times \chi\nu^{-1}$, then $\Theta(\pi)$ is a representation of length two whose unique irreducible quotient is $\chi\circ\det_{\GL_4}$;
    \item otherwise, $\Theta(\pi)$ is the unique irreducible quotient of $\pi\nu^{1/2}\times \pi\nu^{-1/2}$. 
\end{itemize}
\end{rmk}

\end{prop}

\bigskip

\noindent\textbf{Acknowledgment} 
W.T.Gan is partially supported by the Tan Chin Tuan Centennial Professorship at NUS. G. Savin is partially supported by a Simons Foundation gift 946504. P. Baki\'{c} is supported in part by the Croatian Science Foundation grant HRZZ-IP-2022-10-4615.

\end{document}